\documentclass[10pt, reqno]{amsart}
\numberwithin{equation}{section}
\usepackage{amssymb}
\usepackage{hyperref}
\usepackage{cite}

\usepackage{tikz, pgfplots}
\usetikzlibrary{shapes,snakes}
\usetikzlibrary{patterns}

\newcommand{\R}{\mathbb{R}}
\newcommand{\C}{\mathbb{C}}

\newcommand{\wh}[1]{\widehat{#1}}

\newcommand{\eps}{\varepsilon}

\newcommand{\nsc}{|\nabla|^{s_c}} 

\newcommand{\N}{\mathcal{N}}

\newcommand{\M}{\mathcal{M}}
\renewcommand{\S}{\mathcal{S}}

\newcommand{\F}{\mathcal{F}}

\newcommand{\norm}[1]{\left\lVert #1\right\rVert}  

\newcommand{\Z}{\mathbb{Z}}

\def\({\left(} %
\def\){\right)} %
\def\le{\leqslant} %

\newtheorem{theorem}{Theorem}[section]

\newtheorem{lemma}[theorem]{Lemma}

\newtheorem{corollary}[theorem]{Corollary}

\newtheorem{proposition}[theorem]{Proposition}

\theoremstyle{definition}
\newtheorem{definition}[theorem]{Definition}
\newtheorem{remark}[theorem]{Remark}

\theoremstyle{remark}

\newcommand{\qtq}[1]{\quad\text{#1}\quad}

\def\hsc{\dot H_x^{s_c}}

\begin{document}

\title[Radial mass-subcritical NLS]{The radial mass-subcritical NLS in negative order Sobolev spaces}

\author[R. Killip]{Rowan Killip}
\address{Department of Mathematics, University of California Los Angeles, Los Angeles, CA}
\email{killip@math.ucla.edu}

\author[S. Masaki]{Satoshi Masaki}
\address{Department of Systems Innovation, Graduate School of Engineering Sciences, Toyonaka, Osaka, Japan}
\email{masaki@sigmath.es.osaka-u.ac.jp}

\author[J. Murphy]{Jason Murphy}
\address{Department of Mathematics and Statistics, Missouri S\&T, Rolla, MO}
\email{jason.murphy@mst.edu}

\author[M. Visan]{Monica Visan}
\address{Department of Mathematics, University of California Los Angeles, Los Angeles, CA}
\email{visan@math.ucla.edu}

\begin{abstract}
We consider the mass-subcritical NLS in dimensions $d\geq 3$ with radial initial data. 
In the defocusing case, we prove that any solution that remains bounded in the critical Sobolev space throughout its lifespan must be global and scatter.
In the focusing case, we prove the existence of a threshold solution that has a compact flow.
\end{abstract}

\maketitle

\section{Introduction}

We consider the large-data critical problem for the mass-subcritical nonlinear Schr\"odinger equation (NLS):
\begin{equation}\label{nls}
\begin{cases} (i\partial_t + \Delta) u = \mu |u|^p u, \\
u|_{t=0}\in \dot H^{s_c}(\R^d). \end{cases}
\end{equation}
Here $u:\R_t\times\R_x^d\to\C$ for some $d\geq 3$, $\mu\in\{\pm1\}$, and $s_c:=\frac{d}{2}-\frac{2}{p}$.  This problem is \emph{critical} in the sense that the space of initial data is invariant under the rescaling that preserves the class of solutions, namely,
\[
u(t,x)\mapsto \lambda^{\frac{2}{p}}u(\lambda^2t, \lambda x). 
\]

The \emph{mass-critical} NLS corresponds to the case $s_c=0$ (i.e. $p=\frac4d$), in which case the rescaling preserves the conserved {mass} 
\[
M[u(t)] := \int_{\R^d} |u(t,x)|^2\,dx.
\]
\emph{Mass-subcriticality} then refers to the restriction $0<p<\tfrac4d$, in which case $s_c<0$.

The problem \eqref{nls} is known to be ill-posed in $\dot H^{s_c}(\R^d)$ (cf. \cite{CCT1, KPV}).  However, if one restricts to radial (i.e. spherically-symmetric) data, one can recover local well-posedness for $d\geq 3$ and  $p>\frac4{d+1}$ (see \cite{Hidano,GuoWang}).  For technical reasons to be detailed below, our main results do not address this entire range for all dimensions $d\geq 3$.  Instead, we consider $p_0(d)<p<\frac4d$, where
\begin{equation}\label{range}
{
p_0(d) = \begin{cases} 
\frac{15-2d+\sqrt{4d^2+100d+145}}{5(2d-1)} & 3\leq d\leq 8, \\
\frac{4}{d+1} & d\geq 9.\end{cases}
}
\end{equation}

Equation \eqref{nls} with $\mu=1$ is called \emph{defocusing}, while $\mu=-1$ corresponds to the \emph{focusing} case.  Our main result for the defocusing equation is the following theorem.

\begin{theorem}\label{T:main} Let $\mu=1$, $d\geq 3$, and $p_0(d)<p<\frac{4}{d}$.  Suppose $u:I_{\max}\times\R^d\to\C$ is a radial maximal-lifespan solution to \eqref{nls} such that
\[
\| |\nabla|^{s_c} u\|_{L_t^\infty L_x^2(I_{\max}\times\R^d)} < \infty.
\]
Then $I_{\max}=\R$ and $u$ scatters in both time directions, that is, there exist $u_\pm\in \dot H^{s_c}_x$ such that
\[
\lim_{t\to\pm\infty} \| u(t) - e^{it\Delta}u_{\pm}\|_{\dot H_x^{s_c}(\R^d)} = 0.
\]
\end{theorem}

\begin{remark}  See Definition~\ref{D:solution} for the precise notion of solution used here.  The scattering result in Theorem~\ref{T:main} is a consequence of the fact that the solution $u$ will be shown to satisfy critical space-time bounds of the form
\[
\|u\|_{S(\R)} \leq C\bigl(\| |\nabla|^{s_c}u\|_{L_t^\infty L_x^2(\R\times\R^d)}\bigr)
\]
for some function $C(\cdot)$, where $\|u\|_{S(\R)}$ denotes the scattering norm defined in Section~\ref{S:FS}. 
\end{remark}

Theorem~\ref{T:main} fits in the context of recent conditional scattering results for NLS and other dispersive equations; see, for example, \cite{DMMZ, Gao, KenMer2, KV4, KMMV, LuZheng, MMZ, Mur1, Mur2, Mur3, XieFan, Tengfei}.  These results, in turn, were inspired by the global well-posedness and scattering results obtained for the mass- and energy-critical NLS (i.e. \eqref{nls} with $s_c\in\{0,1\}$); see \cite{Bou1, CKSTT, Dod1, Dod2, Dod3, Dod4, Dod5, Gri, KenMer, KTV, KV2, KV3, RV, Tao1, TVZ, Visan, Vis1, Vis2}. 

We will prove Theorem~\ref{T:main} via the concentration-compactness approach to induction on energy.  As such, this result constitutes the first successful application of this type of analysis in negative-regularity Sobolev spaces.

While the well-posedness theory for the mass-subcritical NLS is well-studied (see, for example, \cite{Cazenave} for a textbook treatment), there are very few long-time large-data critical results available and none in the setting of negative-regularity Sobolev spaces.  Our previous work \cite{KMMV} established a large-data critical result in weighted spaces; see also \cite{Ma4}.  Specifically, in \cite{KMMV} we proved that any solution $u$ to \eqref{nls} with $d\geq 1$ and $\max\{ \frac2d, \frac4{d+2}\}<p<\frac4d$ that satisfies
\[
\| |x|^{|s_c|}e^{-it\Delta}u(t)\|_{L_t^\infty L_x^2}<\infty
\]
on its maximal lifespan must be global and scatter in the sense that $|x|^{|s_c|}e^{-it\Delta}u(t)$ converges in $L^2$ as $t\to\pm\infty$.  In fact, this result was shown to hold in both the defocusing and focusing settings.

We would like to point out that when compared with Theorem~\ref{T:main}, the result in \cite{KMMV} has strictly stronger hypotheses and conclusions.  Indeed, by unitarity of $e^{it\Delta}$, Sobolev embedding, and H\"older's inequality, one has 
\[
\||\nabla|^{s_c}u(t)\|_{L_x^2} \lesssim \| e^{-it\Delta} u(t)\|_{L_x^{\frac{dp}{2},2}} \lesssim \| |x|^{|s_c|}e^{-it\Delta}u(t)\|_{L_x^2}.
\]
However, employing Theorem~\ref{T:main} and persistence of regularity type arguments, we can improve the result in \cite{KMMV} for a range of nonlinearities in the radial defocusing case in the following sense: if $u_0$ is radial with $|x|^{|s_c|}u_0\in L^2$ and the corresponding solution remains bounded in $\dot H^{s_c}_x$ throughout its lifespan, then the solution scatters in the weighted norm.  For more details, see Section~\ref{S:weighted} below.

As mentioned above, we prove Theorem~\ref{T:main} via the concentration-compactness approach to induction on energy.  To this end, we introduce the following quantity (cf. \cite{KV2,Ma3}):
\begin{equation}\label{Ec}
E_c :=\inf \limsup_{t\uparrow \sup I_{\max}}\|u(t)\|_{\dot{H}^{s_c}_x},
\end{equation}
where the infimum is taken over all radial maximal-lifespan solutions that do not scatter forward in time.  Note that the small-data theory implies $E_c\in(0,\infty]$ (cf. Corollary~\ref{C:small-data}). Theorem~\ref{T:main} may then be rephrased simply as $E_c=\infty$ for the appropriate range of $d,p$.  The strategy of proof will be to derive a contradiction under the assumption $E_c<\infty$. 

The advantage of working with the quantity $E_c$ is that it allows us to treat the defocusing and focusing problems in parallel. Note that in the focusing case, one does have $E_c<\infty$.  Indeed, if $Q$ is the ground state solution to $-\Delta Q + Q = Q^{p+1}$, then $e^{it}Q(x)$ is a global non-scattering radial solution to \eqref{nls} with $\mu=-1$ that is uniformly bounded in $\dot H^{s_c}_x$.  In particular, $E_c\leq \|Q\|_{\dot H^{s_c}}$.  We will prove that in this case, there exist minimizers to \eqref{Ec}.

\begin{theorem}\label{T:main2}
Let $\mu=-1$, $d\geq 3$, and $\frac{4}{d+1}<p<\frac{4}{d}$.  Then $E_c\leq \|Q\|_{\dot H^{s_c}}$, and there exists a solution that attains $E_c$.
\end{theorem}

Moreover, we will prove that {
there exists a `soliton-like' minimizer of \eqref{Ec}}; see Theorem~\ref{T:main3} below.  However, the connection between these solutions and the ground state $Q$ (if any) is not clear at this time.  We discuss a related question after the statement of Theorem~\ref{T:main3} below. 

In both the defocusing and focusing cases, we will show that if $E_c<\infty$, then there exist almost periodic solutions that attain $E_c$.

\begin{definition}[Almost periodic]\label{D:AP} A non-zero radial solution $u:I\times\R^d\to\C$ to \eqref{nls} is \emph{almost periodic} if $u\in L_t^\infty \dot H_x^{s_c}(I\times\R^d)$ and there exist functions $N:I\to(0,\infty)$ and $C:(0,\infty)\to(0,\infty)$ such that for any $\eta>0$, 
\[
\sup_{t\in I} \int_{|x|>\frac{C(\eta)}{N(t)}} \bigl| |\nabla|^{s_c}u(t,x)\bigr|^2\,dx + \int_{|\xi|>C(\eta)N(t)} \bigl| |\xi|^{s_c}\wh{u}(t,\xi)\bigr|^2\,d\xi < \eta. 
\]
\end{definition}

We then have the following:

\begin{theorem}[Minimal counterexamples]\label{T:reduction}
Let $\mu\in\{\pm1\}$, $d\geq 3$, and $\frac{4}{d+1}<p<\frac4d$.  If $E_c<\infty$, then there exists a radial almost periodic solution $u:I\times\R^d\to\C$ that attains $E_c$ and fits into one of the following scenarios:
\begin{itemize}
\item Self-similar scenario: $I=(0,\infty)$ and $N(t)=t^{-\frac12}$.
\item Cascade scenario: $I=\R$, $\sup_{t\in\R}N(t)\leq 1$, and there exists a sequence of times along which $N(t)\to 0$. 
\item Soliton scenario: $I=\R$ and $N(t)\equiv 1$. 
\end{itemize}
 \end{theorem}

Note that Theorem~\ref{T:reduction} readily implies Theorem~\ref{T:main2}, since $E_c\leq\|Q\|_{\dot H^{s_c}_x}$ in the focusing case. The existence of an almost periodic solution that attains $E_c$ follows along standard lines (see e.g. \cite{KV2, KVClay}); however, as we are working in negative-regularity Sobolev spaces, some details of the argument change, particularly those pertaining to concentration compactness.  We discuss the relevant details in Section~\ref{S:reduction}.  Once the existence of a minimal almost periodic solution is established, the arguments in \cite{KTV, TVZ} carry over to further reduce attention to the three specific solutions recorded above.

The proof of Theorem~\ref{T:main} therefore reduces to precluding the possibility of these three scenarios in the defocusing case.  Our arguments are similar in spirit to those originally used to treat the radial mass-critical problem (cf. \cite{KTV, TVZ}).  The key to precluding each of these three scenarios is to exploit almost periodicity in order to establish additional regularity.  For the self-similar case, it is enough to show that solutions belong to $L^2$.  Indeed, $L^2$ solutions to the mass-subcritical NLS are automatically global, while self-similar solutions blow up at $t=0$.  In the remaining two scenarios, we are able to prove that the solution is in fact bounded in $H^{\frac12}_x$; it is here that we encounter the restriction $p>p_0(d)$. This allows us to prove that cascade solutions have zero mass, yielding a contradiction.  In fact, as the arguments precluding the first two scenarios are equally valid in the focusing case, we deduce the following:

\begin{theorem}[Existence of a soliton-like minimizer]\label{T:minimizer}
Let $\mu\in\{\pm1\}$, $d\geq 3$, and $p_0(d)<p<\frac4d$.  If $E_c<\infty$ then there exists a radial almost-periodic solution $u_c:\R\times\R^d\to\C$ such that
\begin{itemize}
\item $\limsup_{t\to\infty} \norm{u_c(t)}_{\dot{H}^{s_c}_x}=E_c$,
\item $u_c \in C \cap L^\infty (\R;H^{s_c}_x\cap H^{\frac12}_x)$,
\item $N(t)\equiv 1$. In particular, the orbit $\{u_c(t)\}_{t\in\R}$ is precompact in $\dot{H}^{s_c}_x$.
\end{itemize}
\end{theorem}

In the defocusing case, we can preclude the possibility of a solution as in Theorem~\ref{T:minimizer} and hence conclude the proof of Theorem~\ref{T:main}.  To do this, we utilize the Lin--Strauss Morawetz inequality, which relies on the defocusing nature of the nonlinearity.  While we can rule out self-similar and cascade solutions for the focusing problem, we cannot rule out solitons.  Thus we arrive at the following extension of Theorem~\ref{T:main2}, which is our last result for the focusing problem.

\begin{theorem}\label{T:main3} Let $\mu=-1$, $d\geq 3$, and $p_0(d)<p<\frac4d$.  Then $E_c\leq \|Q\|_{\dot H^{s_c}}$ and there exists an almost-periodic solution to \eqref{nls} that attains $E_c$ and has the properties given in Theorem~\ref{T:minimizer}. 
\end{theorem}

We emphasize that Theorem \ref{T:main3} is a true existence theorem. In previous studies of the focusing mass- and energy-critical NLS (e.g. \cite{KV2, Dod5}), the above approach is used to yield a contradiction to the stronger and \emph{false} assumption that $E_c$ is strictly smaller than the size of the ground state measured in the appropriate critical Sobolev norm.  In our case, we do not know whether or not the identity
\begin{equation}\label{E:conj}
E_c=\|Q\|_{\dot{H}^{s_c}_x}
\end{equation}
is true. Whether or not this holds in the mass-subcritical case is an interesting open question.

As mentioned above, the analogue of \eqref{E:conj} was proved by contradiction for the focusing mass- and energy-critical problems.  The key ingredient to preclude the soliton scenario in those settings is a variational characterization of ground states.  Unfortunately, there is no characterization of the ground state to the mass-subcritical problem in the negative-regularity Sobolev space $\dot{H}^{s_c}_x$.

One reason to believe \eqref{E:conj} might hold is that it is consistent with the soliton resolution conjecture, which states that a generic solution to \eqref{nls} decomposes into a sum of solitons plus radiation.  This implies that a generic non-scattering solution will contain at least one soliton component.  As one expects the soliton components and the radiation to be mutually asymptotically orthogonal around $t=\sup{I_{\max}}$, one might expect that $E_c$ is attained by the pure one-soliton solution, which would imply \eqref{E:conj}.  Even belief in this conjecture does not fully settle the matter, since it does not address the behavior of certain non-generic solutions, such as excited-state solitons or other threshold behaviors.  It is equally possible that the minimizer in Theorem~\ref{T:main3} is such an exceptional solution. 

In fact, there are other reasons to believe that \eqref{E:conj} may fail.  In \cite{Ma1, Ma2, Ma4}, it is shown that there exists a threshold solution in the critical weighted space (which is contained in $\dot H^{s_c}_x$) that has minimal weighted norm at a fixed time among all non-scattering solutions.  The threshold is strictly smaller than the weighted norm of the ground state.  However, this does not necessarily imply failure of \eqref{E:conj}, as minimality at a fixed time is not always equivalent to minimality around $t=\sup I_{\max}$ (see \cite{Ma3}).  For further analysis of this type of minimization problem for mass-subcritical dispersive equations, we refer the reader to \cite{Ma1,Ma2,Ma3,Ma4,MS,MS2}.

The rest of this paper is organized as follows.  In Section~\ref{S:notation} we set up notation and collect some useful lemmas.  These include the radial Strichartz estimates (Section~\ref{S:RS}), which play a key role in the analysis throughout the paper.  In Section~\ref{S:reduction}, we discuss the proof of Theorem~\ref{T:reduction}.  The general scheme is well-established; thus we focus on those parts of the argument that are new in our setting.  In Section~\ref{S:SS} we rule out the self-similar scenario.  In Section~\ref{S:AR} we establish additional regularity for the cascade and soliton scenarios.  It is in this section that we encounter the technical restriction appearing in \eqref{range}.  Finally, in Section~\ref{S:CS} we rule out the cascade scenario for $\mu=\pm1$ and the soliton scenario for $\mu=1$, thereby completing the proofs of Theorems~\ref{T:main} and \ref{T:main2}.

\subsection*{Acknowledgements} 
S.M. was supported by the Sumitomo Foundation Basic Science Research
Projects No.\ 161145 and by JSPS KAKENHI Grant Numbers JP17K14219, JP17H02854, and JP17H02851. R. K. was supported, in part, by NSF grant DMS-1600942, J.M. by grant DMS-1400706, and M. V. by grant DMS-1500707.

\section{Notation and useful lemmas}\label{S:notation} For non-negative $X$ and $Y$, we write $X\lesssim Y$ to denote $X\leq C Y$ for some $C>0$.  If $X\lesssim Y\lesssim X$, we write $X\sim Y$.  The dependence of implicit constants on parameters is indicated by subscripts; for example, $X\lesssim_u Y$ denotes $X\leq CY$ for some $C=C(u)$.  We write $a'\in[1,\infty]$ for the dual exponent to $a\in [1,\infty]$, that is, the solution to $\tfrac{1}{a}+\tfrac{1}{a'}=1$.  We write $x\pm$ to denote $x\pm\eps$ for any small $\eps>0$.

For a monomial $X$ we write $\text{\O}(X)$ to denote a finite linear combination of products of the factors of $X$, where Littlewood--Paley projections (see below) and/or complex conjugation may be applied in each factor.  We write $\text{\O}(X+Y)=\text{\O}(X)+\text{\O}(Y)$. 

We will frequently encounter weighted Lebesgue spaces.  At times it will be convenient to employ the following notation:
\[
L_x^{r;\alpha}(\R^d):=L_x^r(\R^d;|x|^\alpha\,dx). 
\]
We can then define $L_t^q L_x^r(I\times\R^d)$ and $L_t^q L_x^{r;\alpha}(I\times\R^d)$ in the usual fashion.  For example, if $q,r<\infty$, then
\begin{equation}\label{weight-notation}
\| u\|_{L_t^q L_x^{r;\alpha}(I\times\R^d)} = \biggl(\int_I \biggl(\int_{\R^d} |u(t,x)|^r |x|^{\alpha}\,dx\biggr)^{\frac{q}{r}}\,dt\biggr)^{\frac{1}{q}} = \| |x|^{\frac{\alpha}{r}} u\|_{L_t^q L_x^r(I\times\R^d)}. 
\end{equation}
When $q=r$, we use the shorthand
\[
\|u\|_{L_{t,x}^{q;\alpha}} = \|u\|_{L_t^q L_x^{q;\alpha}}. 
\]

We need the following Gronwall-type inequality from \cite{KVZ}.

\begin{lemma}[Acausal Gronwall inequality]\label{L:AG} Let $r\in(0,1)$ and $K\geq 4$.  Suppose $\{b_k\}$ is a bounded non-negative sequence and $x_k\geq 0$ satisfies
\[
x_k \leq \begin{cases} b_k & \text{ if} \quad 0 \leq k< K \\ b_k + \sum_{\ell=0}^{k-K}r^{k-\ell}x_\ell & \text{ if} \quad k\geq K.\end{cases}
\]
Then
\[
x_k\lesssim \sum_{\ell=0}^k r^{k-\ell}\exp\bigl\{ \tfrac{\log(K-1)}{K-1}(k-\ell)\}b_\ell. 
\]
\end{lemma}

\subsection{Harmonic analysis tools} 

We will work with the following definition of the Fourier transform:
\[
\F f(\xi) = \wh{f}(\xi) = (2\pi)^{-\frac{d}{2}}\int_{\R^d} e^{-ix\xi}f(x)\,dx. 
\]
For $s\in\R$, $|\nabla|^s$ is defined as the Fourier multiplier operator with multiplier $|\xi|^s$.

Let $\varphi$ be a radial bump function on $\R^d$ supported on $\{ |\xi| \leq \tfrac{5}{3}\}$ and equalling one on $\{|\xi|\leq \frac{4}{3}\}$.  For $N\in 2^{\mathbb{Z}}$ we define the Littlewood--Paley projection operators
\[
\wh{P_{\leq N}f}(\xi)=\varphi(\tfrac{\xi}{N})\wh{f}(\xi),\quad \wh{P_N f}(\xi) = \bigl[\varphi(\tfrac{\xi}{N})-\varphi(\tfrac{2\xi}{N})\bigr]\wh{f}(\xi), \qtq{and} P_{>N}f= f-P_{\leq N}f.
\]
We will often write $f_N=P_N f$ and so on.

These operators are bounded on the usual Lebesgue and Sobolev spaces; in addition, they satisfy the following Bernstein estimates.

\begin{lemma}[Bernstein] Let $1\leq r\leq q\leq\infty$ and $s\geq 0$.  Then
\begin{align*}
\| |\nabla|^{\pm s} P_{ N}u\|_{L_x^r(\R^d)}& \lesssim N^{\pm s} \|P_{N}u\|_{L_x^r(\R^d)}, \\
\|P_{>N}u\|_{L_x^r(\R^d)}&\lesssim N^{-s}\||\nabla|^s P_{>N}u\|_{L_x^r(\R^d)}, \\
\| P_{\leq N} u\|_{L_x^q(\R^d)} & \lesssim N^{\frac{d}r-\frac{d}q}\|P_{\leq N}u\|_{L_x^r(\R^d)}.
\end{align*}
\end{lemma}

We remind the reader that a function $\omega:\R^d\to\R$ is an $A_q$ weight for $1<q<\infty$ precisely when the Hardy--Littlewood maximal function is bounded on $L^q(\R^d;\omega(x)\,dx)$.  For example, the function $\omega(x) = |x|^a$ is an $A_q$ weight for $1<q<\infty$ if and only if 
\[
-d<a<d(q-1).
\] 
Further, if $\omega$ is an $A_q$ weight, one also has the Littlewood--Paley square function estimate
\[
\| f\|_{L^q(\omega(x)\,dx)} \sim \biggl\| \biggl(\sum_{N\in 2^\Z} |P_N f|^2 \biggr)^{\frac12}\biggr\|_{L^q(\omega(x)\,dx)}.
\]
For this and much more see \cite[Chapter V]{Stein}).

We will need the following weighted Bernstein estimate.

\begin{lemma}[Weighted Bernstein]\label{L:WB} Suppose $1<q<\infty$ and $-d<a<d(q-1)$.  Then $NP_N\Delta^{-1}\nabla$
is bounded on $L^q(|x|^a\,dx)$ uniformly for $N\in 2^\Z$.  Consequently,
\[
\| |x|^{\frac{a}{q}} P_{>N} f\|_{L_x^q(\R^d)} \lesssim N^{-1} \| |x|^{\frac{a}{q}} \nabla f\|_{L_x^q(\R^d)}. 
\]
\end{lemma}

\begin{proof} The operator $NP_N\Delta^{-1}\nabla$ is a convolution operator with an $L^1$-normalized Schwartz function, and hence it is controlled pointwise by the Hardy--Littlewood maximal function.  Thus boundedness of this operator on $L^q(|x|^a\,dx)$ follows from the fact that $|x|^a$ is an $A_q$ weight.  To deduce the estimate, write $P_Mf=\Delta^{-1}\nabla\cdot\nabla P_Mf$ and sum over $M>N$. 
\end{proof}

\subsection{The linear Schr\"odinger equation} The Schr\"odinger group 
\[
e^{it\Delta} = \F^{-1} e^{-it|\xi|^2} \F
\]
is a convolution operator with kernel
\begin{equation}\label{propagator}
[e^{it\Delta} f](x,y) = (4\pi it)^{-\frac{d}{2}}e^{\frac{i|x-y|^2}{4t}},\quad t\neq 0.
\end{equation}

From these two identities one can deduce
\[
\| e^{it\Delta} f\|_{L_x^2} \equiv \|f\|_{L_x^2}\quad \qtq{and}\quad  \| e^{it\Delta} f\|_{L_x^\infty} \lesssim |t|^{-\frac{d}{2}}\|f\|_{L_x^1} \quad \text{for $t\neq 0$}.
\]
Interpolation then yields the dispersive estimate
\[
\|e^{it\Delta}f\|_{L_x^r} \lesssim |t|^{-(\frac{d}{2}-\frac{d}{r})} \|f\|_{L_x^{r'}}\qtq{for} 2\leq r\leq\infty \qtq{and} t\neq 0.
\]

These estimates are used to establish the standard Strichartz estimates:

\begin{proposition}[Strichartz estimates, \cite{GinVel, KeeTao, Str}] For $d\geq 3$ and  $2\leq q_j,r_j\leq\infty$ satisfying $\tfrac{2}{q_j}+\tfrac{d}{r_j}=\tfrac{d}{2},$
one has
\begin{align*}
\| e^{it\Delta} f \|_{L_t^{q_1} L_x^{r_1}(\R\times\R^d)} & \lesssim \|f\|_{L_x^2(\R^d)}, \\
\biggl\| \int_{t_0}^t e^{i(t-s)\Delta}F(s)\,ds\biggr\|_{L_t^{q_1} L_x^{r_1}(I\times\R^d)} & \lesssim \|F\|_{L_t^{q_2'}L_x^{r_2'}(I\times\R^d)},
\end{align*}
for any $t_0\in \overline{I}\subset\R$. 
\end{proposition}

We will also need the following bilinear Strichartz estimate (cf. \cite[Corollary~4.19]{KVClay}). 

\begin{lemma}[Bilinear Strichartz]\label{L:BS} For $d\geq 3$ and $s_c\in(-\tfrac12,0)$, 
\[
\| [e^{it\Delta}f_{\leq M}][e^{it\Delta}g_{>N}]\|_{L_{t,x}^2(\R\times\R^d)} \lesssim M^{\frac{d-1}{2}-s_c}N^{-(\frac12+s_c)}\|f\|_{\dot H_x^{s_c}(\R^d)}\|g\|_{\dot H_x^{s_c}(\R^d)}. 
\]
\end{lemma}

In particular, we need the following corollary of Lemma~\ref{L:BS}.   

\begin{corollary}\label{C:BS} Let $d\geq 3$.  Define
\[
\|u\|_* = \|u(t_0)\|_{\dot H_x^{s_c}(\R^d)} + \|(i\partial_t+\Delta)u\|_{N(I)}
\]
with $t_0\in I\subset\R$ and $N(I)$ as in \eqref{spaces}.  Then for spherically-symmetric functions $u$ and $v$ we have
\[
\|u_{\leq M} v_{>N}\|_{L_{t,x}^2(I\times\R^d)} \lesssim M^{\frac{d-1}{2}+|s_c|}N^{-(\frac12-|s_c|)} \|u_{\leq M}\|_* \|v_{>N}\|_*. 
\]
\end{corollary}

\begin{proof} The proof is almost identical to the one appearing in \cite[Lemma~2.5]{Visan}.  There are two small changes:  (i) We use frequency-localized solutions and always put the solutions in $\dot H^{s_c}$, yielding the appropriate powers of $M$ and $N$.  (ii) We use the dual radial Strichartz estimate
\[
\biggl\| |\nabla|^{s_c}\int_\R e^{-is\Delta}F(s)\,ds\biggr\|_{L_x^2(\R^d)} \lesssim \|F\|_{N(\R)}
\]
instead of the standard dual Strichartz estimate (see Section~\ref{S:RS}). \end{proof}

\subsection{Radial Strichartz estimates}\label{S:RS} We will rely heavily on improved Strichartz estimates that are available in the radial setting.  We begin by recording a weighted dispersive estimate.

\begin{lemma}\label{L:WDE} Let $d\geq 2$, $r\in(2,\infty)$, and
\[
0\leq\alpha\leq \tfrac{d-1}{2}(1-\tfrac2r).
\]
For any spherically-symmetric function $f$,
\[
\| |x|^\alpha e^{it\Delta}(|\cdot|^\alpha f) \|_{L_x^r(\R^d)} \lesssim |t|^{-(\frac{d}{2}-\frac{d}{r})+\alpha}\|f\|_{L_x^{r'}(\R^d)}.
\]
\end{lemma}

\begin{proof} Writing $P_{\text{rad}}$ for the projection onto radial functions, we have the kernel estimate
\[
|e^{it\Delta}P_{\text{rad}}(x,y)| \lesssim \min\{|t|^{-\frac{d}{2}},(|x||y|)^{-\frac{d-1}{2}}t^{-\frac12}\} \lesssim t^{-\frac{d}{2}+\theta(\frac{d-1}{2})}(|x| |y|)^{-\theta(\frac{d-1}{2})}
\]
for any $\theta\in[0,1]$ (cf. \eqref{propagator} and \cite[(2.9)]{KVZ}).  This implies the desired estimate with $r=\infty$; the remaining estimates follow from interpolation with $\|e^{it\Delta}\phi\|_{L_x^2} = \|\phi\|_{L_x^2}$. 
\end{proof}

To obtain a wide range of radial Strichartz estimates, we will interpolate between several existing results, which we record in the following proposition:

\begin{proposition}[Radial Strichartz estimates]\label{P:rad1} Let $d\geq 2$, let $\alpha,s\in\R$, and let $q,r\in[2,\infty]$ satisfy the following scaling condition:
\[
\tfrac{2}{q}+\tfrac{d}{r}+\alpha = \tfrac{d}{2}-s. 
\]
For any spherically-symmetric function $f$ we have 
\[
\| |x|^{\alpha} e^{it\Delta} f \|_{L_t^q L_x^r(\R\times\R^d)} \lesssim \| f\|_{\dot H_x^s(\R^d)}
\]
in each of the following three scenarios:\\[1mm]
(i) $(q,r)=(4,\infty)$, $\alpha=\frac{d-1}{2}$, and $s=0$. \\[1mm]
(ii) {
$q=r$} and $-\tfrac{d}{q}<\alpha<\tfrac{d-1}{2}-\tfrac{d}{q}$.\\[1mm] 
(iii) $\alpha=0$ and $\tfrac{2}{q} + \tfrac{2d-1}{r}\leq d-\tfrac12$, with strict inequality whenever $\tfrac1q>\tfrac12-\tfrac1{2d+1}$. 
\end{proposition}

\begin{proof} The estimate (i) appears in \cite[Lemma~2.8]{KVZ}.  The estimate (ii) appears in \cite[Theorem~2.1]{Hidano}. For the estimate (iii), see \cite[Theorem~1.5]{GuoWang}. 
\end{proof}

By interpolation, Proposition~\ref{P:rad1} yields estimates for every $(\tfrac{1}{r},\tfrac{1}{q})\in(0,\tfrac12)\times(0,\tfrac12)$ (as well as {
 points} on the boundary of this square).  For each such pair, one can interpolate in two different ways, so as to obtain the largest and smallest possible values of the weight $\alpha$.  The result is the following:

\begin{theorem}[Interpolated radial Stichartz estimates]\label{T:RS} Let $d\geq 2$, let $\alpha,s\in\R$, and let $q,r\in(2,\infty)$ satisfy the following scaling condition:
\begin{equation}\label{scalecond}
\tfrac{2}{q}+\tfrac{d}{r}+\alpha = \tfrac{d}{2}-s.
\end{equation}
For any spherically-symmetric function $f$ we have 
\[
\| |x|^\alpha e^{it\Delta} f\|_{L_t^q L_x^r(\R\times\R^d)} \lesssim \|f\|_{\dot H_x^s(\R^d)},
\]
provided
\begin{equation}\label{alpharange}
\max\bigl\{-\tfrac{d}{r},-\tfrac{d}{q}\bigr\}<\alpha<\min\bigl\{\tfrac{d-1}{2}-\tfrac{d-1}{r},(d-1)[1-\tfrac2q],\tfrac{2d-1}{4}[1-\tfrac2r]-\tfrac1q\bigr\}.
\end{equation}
\end{theorem}

\begin{proof} We define the points
\[
O = (0,0),\ A=(\tfrac12,0),\ B=(\tfrac12,\tfrac12),\ C=(0,\tfrac12),\ D=(\tfrac12-\tfrac1{2d-1},\tfrac12),\ E=(0,\tfrac14) 
\]
and consider the following figure: 

\begin{center}
\begin{tikzpicture}[scale=7.5]
\draw[ thick,->] (0,0) -- (0,0.6) node [left] at (0,0.6) {$1/q$}; 
\draw[ thick,->] (0,0) -- (0.7,0); \node [below] at (0.7,0) {$1/r$}; 
\draw[ thick,-] (0.5,0) -- (0.5,0.5) -- (0,0.5); 
\draw[ thick,-] (0.5, 0) -- (0,0.25) -- (0.3,0.5) -- (0.5, 0); 
\fill (0,0) circle[radius=0.3pt] node[below left] { $O$};
\fill (0.5,0) circle[radius=0.3pt] node[below] { $A$};
\fill (0.5,0.5) circle[radius=0.3pt] node[right] { $B$};
\fill (0,0.5) circle[radius=0.3pt] node[left] { $C$};
\fill (0.3,0.5) circle[radius=0.3pt] node[above] { $D$};
\fill (0,0.25) circle[radius=0.3pt] node[left] { $E$};
\end{tikzpicture}
\end{center}

We first consider the problem of obtaining the largest possible weight.  The interior of the triangle $OAE$ is defined by
\[
\tfrac{1}{q}<\tfrac14-\tfrac{1}{2r}.
\]
The largest possible weight on the line $OE$ is $\tfrac{d-1}{2}$, while the point $A$ requires $\alpha=0$.  This leads to the upper bound
\[
\alpha<\tfrac{d-1}{2}-\tfrac{d-1}{r}
\]
for $(\tfrac1r,\tfrac1q)$ inside this triangle.

On the interior of the triangle $DCE$ we have 
\[
\tfrac{1}{q}>\tfrac14+\tfrac{2d-1}{2(2d-3)}\tfrac1r. 
\]
The largest possible weight is at the point $E$; the points $C$ and $D$ require $\alpha=0$. This leads to the upper bound
\[
\alpha < (d-1)[1-\tfrac2q]
\]
for $(\tfrac1r,\tfrac1q)$ inside this triangle.

On the interior of the triangle $ADE$ we have
\[
\tfrac14-\tfrac1{2r} < \tfrac1q < \min\{\tfrac14+\tfrac{2d-1}{2(2d-3)}\tfrac1r,\tfrac{2d-1}{4}-\tfrac{2d-1}{2}\tfrac{1}{r}\}. 
\]
Again, the largest possible weight is at the point $E$, with the line $AD$ requiring $\alpha=0$.  Interpolating between $E$ and points on $AD$ leads to the upper bound 
\[
\alpha < \tfrac{2d-1}{4}[1-\tfrac2r]-\tfrac1q. 
\]

Finally, on the interior of the triangle $ABD$ we have
\[
\tfrac{1}{q}>\tfrac{2d-1}{4}-\tfrac{2d-1}{2}\tfrac{1}{r}. 
\]
The points $A$ and $D$ require $\alpha=0$, while the largest weight at $B$ is $\alpha<-\tfrac12$. This again leads to the upper bound
\[
\alpha<\tfrac{2d-1}{4}[1-\tfrac2r]-\tfrac1q.
\]

We turn now to the problem of obtaining the smallest possible weight.  We first consider the triangle $ABO$.  The smallest weight at $A$ and $O$ is $\alpha=0$, while the smallest weight at $B$ is $\alpha>-\frac{d}{2}$.  Thus interpolating between $B$ and $AO$ leads to the lower bound $\alpha>-\tfrac{d}{q}$. 

We next consider the triangle $OBC$.  Interpolating between $O$ and $C$, we find the smallest weight on $OC$ is $\alpha=0$.  Thus, interpolating between $B$ and $OC$ leads to the lower bound $\alpha>-\tfrac{d}{r}$.
 
This completes the proof of the theorem.
\end{proof}

We also record the following dual/inhomogeneous Strichartz estimates.

\begin{proposition}[Dual Strichartz estimate]\label{DS}
Let $d\geq 2$ and let $(q,r,\alpha,s)$ satisfy \eqref{scalecond} and \eqref{alpharange} with $q,r\in(2,\infty)$.  Then the estimate
\[
\norm{\int_{\R}e^{-is\Delta} F(s) ds}_{\dot{H}^{-s}_x}\lesssim\| |x|^{-\alpha} F \|_{L_t^{q'} L_x^{r'}(\R\times\R^d)}
\]
holds for any spherically-symmetric function $F$.
\end{proposition}

\begin{proposition}[Inhomogeneous Strichartz estimate] Fix $d\geq 2$.  Let $(q,r,\alpha,s)$ and $(\tilde q,\tilde r,\tilde\alpha,\tilde s)$ satisfy \eqref{scalecond} and \eqref{alpharange} with $q,r, \tilde q, \tilde r\in(2,\infty)$ and $\tilde s =-s$. Then for any $t_0\in\bar I\subset\R$, the estimate 
\[
\biggl\| \int_{t_0}^t e^{i(t-s)\Delta}F(s)\,ds\biggr\|_{L_t^\infty \dot H_x^s} + \biggl\||x|^\alpha \int_{t_0}^t e^{i(t-s)\Delta}F(s)\,ds\biggr\|_{L_t^q L_x^r} \lesssim \| |x|^{-\tilde \alpha}F\|_{L_t^{\tilde q'}L_x^{\tilde r'}}
\]
holds on $I\times\R^d$ for any spherically-symmetric function $F$. \end{proposition}

We conclude with the following frequency-localized inhomogeneous Strichartz estimate.
\begin{proposition}\label{P:NStrichartz} Fix $d\geq 2$. Let $(q,r,\alpha,s)$ and $(\tilde q,\tilde r,\tilde \alpha,\tilde s)$ satisfy \eqref{scalecond} and \eqref{alpharange} with $q,r, \tilde q, \tilde r\in(2,\infty)$.  Then for any $t_0\in\bar I\subset\R$, the estimate 
\begin{align*}
\biggl\| P_N\int_{t_0}^t e^{i(t-s)\Delta}F(s)\,ds\biggr\|_{L_t^\infty \dot H_x^s} + \biggl\| |x|^{\alpha} P_N&\int_{t_0}^t e^{i(t-s)\Delta}F(s)\,ds\biggr\|_{L_t^q L_x^r} \\
&\lesssim N^{s+\tilde s}\| |x|^{-\tilde\alpha} P_N F\|_{L_t^{\tilde q'}L_x^{\tilde r'}}
\end{align*}
holds on $I\times\R^d$ for any spherically-symmetric function $F$.  The operator $P_N$ can be replaced by $P_{>N}$ if $s+\tilde s\leq 0$ and by $P_{\leq N}$ if $s+\tilde s\geq 0$. 
\end{proposition}

\begin{proof}
The first term can be handled by Bernstein and the previous inhomogeneous Strichartz estimate.  Hence, we focus on the second term. Let $\tilde P_N$ denote the fattened Littlewood--Paley operator.  By the Strichartz and the dual Strichartz estimates, we have
\[
\norm{|x|^\alpha \tilde{P}_N \int_\R e^{i(t-s)\Delta} G(s) ds}_{L^q_t L^r_x (\R \times \R^d)}
	\lesssim N^{s+\tilde{s}} \| |x|^{-\tilde{\alpha}} G\|_{L_t^{\tilde{q}'} L_x^{\tilde{r}'}(\R\times\R^d)}
\]
for any spherically-symmetric function $G$. Thus the desired bound follows from the Christ--Kiselev lemma (choosing $G=P_NF$). \end{proof}

\subsection{Function spaces}\label{S:FS} In this section we define some specific function spaces that will be used throughout the paper.  Recall that we are considering \eqref{nls} with $d\geq 3$ and $\tfrac{4}{d+1}<p<\tfrac{4}{d}$. 

First we remark that we may choose an exponent $q$ such that
\begin{equation}\label{q-conditions}
\max\{\tfrac{1}{p+1},\tfrac12-s_c,\tfrac{4}{p(d+2)}\} < \tfrac{2}{q} < \min\{1,\tfrac23\cdot\tfrac{1}{p+1}[\tfrac{5}{2}-s_c]\}
\end{equation}
and define
\[
\gamma=\tfrac{2q}{p} - (d+2)<0. 
\]
To see that we may choose $q$ as in \eqref{q-conditions}, we note that $p>\frac4{d+1}$ is equivalent to $\frac12-s_c<1$.  Also,
\[
\tfrac{2}{3}\cdot\tfrac{1}{p+1}[\tfrac52-s_c]-\tfrac{4}{p(d+2)} = \tfrac{(d-1)(4-p(d-2))}{3(d+2)p(1+p)}>0
\]
provided $p<\tfrac{4}{d-2}$, and 
\[
\tfrac{2}{3}\cdot\tfrac{1}{p+1}[\tfrac52-s_c]-(\tfrac12-s_c) = \tfrac{3(d-1)p^2+(d-5)p-4}{6p(p+1)}>0
\]
whenever $p>\tfrac{4}{d+1}$. Indeed, the larger root of the polynomial appearing in the numerator above is $\tfrac{5-d+\sqrt{-23+38d+d^2}}{6(d-1)}$, which is $\leq \tfrac{4}{d+1}$ (with equality when $d=3$).

\begin{remark} Some of the above restrictions are based on the requirements of the Strichartz estimates of \cite{Hidano} (cf. Proposition \ref{P:rad1} (ii)).
Even though Theorem~\ref{T:RS} provides slightly better estimates, we cannot improve on the restriction $p>\frac{4}{d+1}$.
\end{remark}

Recalling the notation in \eqref{weight-notation}, we now define
\begin{equation}\label{spaces}
\left.\begin{aligned}
S_I(u)&=\|u\|_{S(I)}  = \| u\|_{ L_{t,x}^{q;\gamma}(I\times\R^d)}, \\
\|F\|_{N(I)} & = \| F\|_{L_{t,x}^{\frac{q}{p+1};\gamma}(I\times\R^d)}.
\end{aligned}\qquad\qquad\right\}
\end{equation}
We write $S(I) = \{f:\|f\|_{S(I)}<\infty\}$, and similarly for $N(I)$.  With this notation,
\begin{equation}\label{wp}
\left.\begin{gathered}\!\!\!\!\biggl\| \int_{t_0}^t e^{i(t-s)\Delta}F(s)\,ds\biggr\|_{L_t^\infty \dot H_x^{s_c}(I\times\R^d)} + \biggl\|\int_{t_0}^t e^{i(t-s)\Delta}F(s)\,ds\biggr\|_{S(I)} \lesssim \|F\|_{N(I)}\\
\text{and}\qquad \bigl\| |u|^{p}u \bigr\|_{N(I)} \leq  \bigl\| u \bigr\|_{S(I)}^{p+1}.
\end{gathered}\ \right\}
\end{equation}

Finally, we remark that $|x|^{\gamma}$ is an $A_q$ weight, as $p<q$ implies $\gamma>-d$.

\subsection{Local well-posedness and stability}\label{S:LWP}  In this section we record local well-posedness and stability results for \eqref{nls}.  These results are valid in the range $d\geq 3$ and $\tfrac{4}{d+1}<p<\tfrac{4}{d}$.  As the arguments are standard, we simply state the results.  The key ingredient is the Strichartz estimate \eqref{wp}. 

\begin{definition}[Solution]\label{D:solution} A function $u:I\times\R^d\to\C$ is a \emph{solution} to \eqref{nls} if it belongs to $C_t \dot H_x^{s_c}(K\times\R^d)\cap S(K)$ for any compact $K\subset I$ and obeys the Duhamel formula
\[
u(t) = e^{it\Delta} u_0 - i\mu \int_0^t e^{i(t-s)\Delta}(|u|^p u)(s)\,ds
\]
for any $t\in I$.  We call $I$ the \emph{lifespan} of $u$.  We call $u$ a \emph{maximal-lifespan solution} if it cannot be extended to any strictly larger interval.  We call $u$ \emph{global} if $I=\R$. 
\end{definition}

\begin{theorem}[Local well-posedness]\label{T:LWP} For any radial $u_0\in \dot H^{s_c}$, there exists a unique maximal-lifespan solution $u:I_{\max}\times\R^d\to\C$ to \eqref{nls} with $u(0)=u_0$. 

If $T_{\max}<\infty$ then $\lim_{T\to T_{\max}}S_{[t_0,T]}(u)=\infty$ for any $t_0\in I_{\max}$.  An analogous result holds backward in time.

The solution scatters in $\dot H^{s_c}$ forward in time if and only if $S_{[t_0,T_{\max})}(u)<\infty$ for some $t_0\in I_{\max}$, in which case $T_{\max}=\infty$. An analogous result holds backward in time. 
\end{theorem}

We also have the standard stability result. 

\begin{theorem}[Stability]\label{T:stab} Let $I$ be an interval and $\tilde u:I\times\R^d\to\C$ be a radial solution to
\[
(i\partial_t+\Delta) \tilde u = \mu |\tilde u|^p \tilde u + e
\]
for some function $e:I\times\R^d\to\C$.  Suppose that 
\[
\|\tilde u\|_{L_t^\infty \dot H_x^{s_c}(I\times\R^d)} + \|\tilde u\|_{S(I)}\leq M.
\]
Fix $t_0\in I$ and $u_0\in \dot H^{s_c}_x$.  There exists $\eps_1=\eps_1(M)>0$ such that if
\[
\|e^{i(t-t_0)\Delta}(\tilde u(t_0) - u_0)\|_{S(I)} < \eps \qtq{and}
\|e\|_{N(I)}<\eps
\]
for $0<\eps<\eps_1$, then there exists a unique radial solution $u:I\times\R^d\to\C$ to \eqref{nls} with $u(t_0)=u_0$ satisfying $\|u-\tilde u\|_{S(I)}\lesssim \eps.$
\end{theorem}

Finally, applying Theorem~\ref{T:stab} with $\tilde u(t)=e^{it\Delta}u_0$ leads to the following small-data scattering result.

\begin{corollary}\label{C:small-data} Let $u_0\in \dot H^{s_c}(\R^d)$ be a radial function.  If $\|u_0\|_{\dot H^{s_c}}$ is sufficiently small, then the solution $u$ to \eqref{nls} with $u(0)=u_0$ is global and $\|u\|_{S(\R)}<\infty$.  In particular, $u$ scatters in both time directions.  
\end{corollary}

\subsection{Scattering in weighted spaces}\label{S:weighted}

As mentioned in the introduction, Theorem~\ref{T:main} yields an improvement to our previous results in \cite{KMMV} for a range of nonlinearities in the radial defocusing case.  Specifically, in certain cases we can show that solutions with radial data in the weighted space
$$
\F \dot H^{|s_c|}=\{f:|x|^{|s_c|}f\in L^2\}
$$
that remain bounded in $\dot H^{s_c}$ actually scatter in the stronger norm $\F \dot H^{|s_c|}$, in the sense that $e^{-it\Delta}u(t)$ converges in $\F\dot H^{|s_c|}$ as $t\to\pm\infty$.  Note that the embedding $\F \dot H^{|s_c|}\hookrightarrow \dot H^{s_c}$ for $-\frac d2<s_c\leq 0$ follows from Hardy's inequality and Plancherel.

As a simple example, let us demonstrate this fact when $d=3$ and $p_0(3)<p<\tfrac43$.  We suppose that $u_0\in \F\dot H^{|s_c|}$ is radial, and that the solution $u$ to \eqref{nls} with $u(0)=u_0$ remains uniformly bounded in $\dot H^{s_c}$ throughout its lifespan.  Theorem~\ref{T:main} then implies that $u$ is global, with $u\in S(\R)$.  A further application of Strichartz (Theorem~\ref{T:RS}) shows that $u$ also belongs to the critical space $L_t^{2p-}L_x^{3p+}(\R\times\R^3)$.  To use this norm without weights requires only $p>\tfrac{16}{15}$ (see \eqref{alpharange}). 

We will show that $\{e^{-it\Delta}u(t)\}$ is Cauchy in $\F\dot H^{|s_c|}$ as $t\to\infty$.  The case $t\to-\infty$ is similar.  To this end, we recall the operator 
\[
J^s(t) = e^{it\Delta} |x|^s e^{-it\Delta} = e^{i|x|^2/4t}(-4t^2\Delta)^{s/2} e^{-i|x|^2/4t},
\]
which featured prominently in \cite{KMMV}.  We will first show that $J^{|s_c|}u$ remains bounded in $L^2$.  To see this, we use the Strichartz estimates in \cite{KMMV} (see e.g. Proposition~2.6 therein) and the Duhamel formula
\[
J^{|s_c|}(t)u(t) = e^{i(t-t_0)\Delta}J^{|s_c|}(t_0)u(t_0)-i\int_{t_0}^t e^{i(t-s)\Delta}J^{|s_c|}(s)\bigl(|u|^p u\bigr)(s)\,ds.
\]
Writing $\tilde u(t)=e^{-i|x|^2/4t}u(t)$, we get the following estimate on any spacetime slab $(t_0,t_1)\times\R^3$:
\begin{align*}
\| J^{|s_c|}u \|_{L_t^\infty L_x^2} & \lesssim \|e^{-it_0\Delta}u(t_0)\|_{\F\dot H^{|s_c|}} +  \| |t|^{|s_c|}|\nabla|^{|s_c|}(|\tilde u|^p \tilde u)\|_{L_t^{2-,2}L_x^{\frac65+}} \\
& \lesssim \|J^{|s_c|}u(t_0)\|_{L^2} + \|u\|_{L_t^{2p-,2p} L_x^{3p+}}^p \| |t|^{|s_c|}|\nabla|^{|s_c|}\tilde u\|_{L_t^\infty L_x^2} \\
& \lesssim \|J^{|s_c|}u(t_0)\|_{L^2}+ \|u\|_{L_t^{2p-} L_x^{3p+}}^p \| J^{|s_c|}u\|_{L_t^\infty L_x^2},
\end{align*}
where $L_t^{a,b}$ denotes the Lorentz space.  Combining this estimate with a bootstrap argument (splitting $\R$ into finitely many intervals on which the $L_t^{2p-}L_x^{3p+}$-norm is small) yields boundedness of $J^{|s_c|}u$ in $L^2$.  Applying the same estimates once more, we now get
\[
\|e^{-it\Delta}u(t)-e^{-is\Delta}u(s)\|_{\F\dot H^{|s_c|}}\lesssim \|u\|_{L_t^{2p-}L_x^{3p+}([s,t]\times\R^3)}^p\| J^{|s_c|}u\|_{L_t^\infty L_x^2} \to 0
\]
as $s,t\to\infty$, which yields the desired result.  

\section{Concentration compactness}\label{S:reduction}

In this section we discuss the proof of Theorem~\ref{T:reduction}.  We denote the $\dot H^{s_c}$-preserving dilation by 
\[
 [D(\lambda)f](x) = \lambda^{-\frac2p}f(\tfrac{x}{\lambda})
\]
and recall the notation
\[
\|f\|_{S(I)}= \| f\|_{L_{t,x}^{q;\gamma}(I\times\R^d)},
\]
with $q$ and $\gamma$ defined in Section~\ref{S:FS}. 

The proof of Theorem~\ref{T:reduction} relies heavily on the following linear profile decomposition.  


\begin{proposition}[Linear profile decomposition]\label{P:LPD} Let $\{u_n\}\subset \dot H^{s_c}$ be a bounded sequence of radial functions. There exist radial $\{\psi_j\}\subset \dot H^{s_c}$, radial $\{R_n^J\}\subset \dot H^{s_c}$, $\{N_n^j\}\subset 2^{\mathbb{Z}}$ and $\{t_n^j\}\subset\R$ such that (passing to a subsequence) the following decomposition holds for each $J\geq 1$:
\[
u_n=\sum_{j=1}^J D(N_n^j)e^{it_n^j\Delta} \psi^j + R_n^J.
\]
Furthermore, this decomposition has the following properties:

\noindent (1) The parameters are asymptotically orthogonal: 
\[
\lim_{n\to\infty} \tfrac{N_n^j}{N_n^k} + \tfrac{N_n^k}{N_n^j} + \bigl| t_n^j - (\tfrac{N_n^j}{N_n^k})^2t_n^k\bigr| = \infty\qtq{for}j\neq k.
\]
\noindent (2) There is asymptotic decoupling of the $\dot H^{s_c}$ norm: for each $J\geq 1$,
\[
\|u_n\|_{\dot H^{s_c}}^2 = \sum_{j=1}^J \|\psi^j\|_{\dot H^{s_c}}^2 + \|R_n^J\|_{\dot H^{s_c}}^2+ o(1)\qtq{as}n\to\infty.
\]
\noindent (3) The remainders vanish in the following sense:
\[
\limsup_{J\to\infty}\limsup_{N\to\infty} \|e^{it\Delta} R_n^J\|_{S(\R)}=0.
\]
\end{proposition}

The proof of Proposition~\ref{P:LPD} follows well-known arguments detailed, for example, in \cite{KVClay}. Hence, we emphasize only the main points and the details that need to be modified in our setting.  The proof proceeds by induction, successively removing packets of concentration from the sequence by appealing to the following proposition.

\begin{proposition}[Inverse Strichartz]\label{P:IS} Assume $\{u_n\}\subset \dot H^{s_c}$ is a sequence of radial functions such that
\[
\|u_n\|_{\dot H^{s_c}} \leq A<\infty \qtq{and}\|e^{it\Delta} u_n\|_{S(\R)}\geq \eps >0. 
\]
Then there exists $\beta=\beta(\eps,A)$, $\{N_n\}\subset 2^{\mathbb{Z}}$, and $\{t_n\}\subset\R$ such that (passing to a subsequence)
\begin{align*}
& [D(N_n)e^{it_n\Delta}]^{-1}u_n\rightharpoonup \psi\qtq{weakly in}\dot H^{s_c}\qtq{as}n\to\infty, \qtq{with} \|\psi\|_{\dot H^{s_c}}\geq\beta. 
\end{align*}
\end{proposition}

In turn, the starting point for Proposition~\ref{P:IS} is the following improved Strichartz estimate, which firstly identifies a scale at which concentration occurs. 

\begin{lemma}[Improved Strichartz]\label{L:IS} For any radial $f\in \dot H^{s_c}$, 
\[
\|e^{it\Delta} f\|_{S(\R)} \lesssim \|f\|_{\dot H^{s_c}}^{\frac2q}\biggl[\sup_{N\in 2^{\mathbb{Z}}} \| e^{it\Delta} P_N f\|_{S(\R)}\biggr]^{1-\frac2q}. 
\]
\end{lemma}

\begin{proof}  Fix an integer $k\geq 2$ with $2(k-1)<q\leq 2k$ and take $r_1>q>r_k$ so that 
\[
\tfrac1{r_1}+\tfrac1{r_k}=\tfrac{2}{q}.
\]
Note that by Strichartz (choosing $r_1$ and $r_k$ close enough to $q$) and Bernstein, we have
\[
\| e^{it\Delta} P_N f \|_{L_{t,x}^{r_j;\gamma}} \lesssim N^{\frac{2}{p}(1-\frac{q}{r_j})}\| |\nabla|^{s_c} P_N f\|_{L_x^2}.
\]
Thus, writing $F(t)=e^{it\Delta}f$ and recalling that $|x|^\gamma$ is an $A_q$ weight, we have by the square function estimate, H\"older, Strichartz, and Bernstein, 
\begin{align*}
\| F\|_{L_{t,x}^{q;\gamma}}^q &\sim \biggl\| \biggl(\sum_M |F_M|^2\biggr)^{\frac12}\biggr\|_{L_{t,x}^{q;\gamma}}^q  \lesssim \sum_{N_1\leq\cdots\leq N_k} \iint \biggl[\prod_{j=1}^k |F_{N_j}|^{\frac{q}{k}}\biggr]|x|^\gamma\,dx\,dt \\
& \lesssim \sum_{N_1\leq\cdots \leq N_k}\biggl(\prod_{j=1,k}\| F_{N_j}\|_{L_{t,x}^{q;\gamma}}^{\frac{q}{k}-1}\|F_{N_j}\|_{L_{t,x}^{r_j;\gamma}}\biggr) \prod_{\ell=2}^{k-1} \|F_{N_\ell}\|_{L_{t,x}^{q;\gamma}}^{\frac{q}{k}} \\
& \lesssim \sup_N \|F_N\|_{L_{t,x}^{q;\gamma}}^{q-2} \sum_{N_1\leq\cdots\leq N_k}(\tfrac{N_1}{N_k})^{\frac2p(1-\frac{q}{r_1})}\| \nsc f_{N_1}\|_{L_x^2}\|\nsc f_{N_k}\|_{L_x^2}.
\end{align*}
The result now follows from an application of Cauchy--Schwarz. \end{proof} 

With Lemma~\ref{L:IS} in place, we can now sketch the proof of Proposition~\ref{P:IS}.

\begin{proof}[Proof of Proposition~\ref{P:IS}] Using Lemma~\ref{L:IS}, we deduce the existence of $N_n\in 2^{\mathbb{Z}}$ such that
\[
\| |x|^{\frac{\gamma}{q}}e^{it\Delta} P_{N_n}u_n\|_{L_{t,x}^q} \gtrsim A(\tfrac{\eps}{A})^{\frac{q}{q-2}}.
\]

As $\tfrac12-s_c<\tfrac2q$, we may choose $\theta$ satisfying
\[
\tfrac{2}{d-1}(\tfrac{d}{2}-s_c-\tfrac2q) < \theta <1. 
\]
Then by H\"older's inequality and Strichartz, 
\begin{align*}
A(\tfrac{\eps}{A})^{\frac{q}{q-2}} & \lesssim \| |x|^{\frac{\gamma}{q\theta}} e^{it\Delta} P_{N_n} u_n \|_{L_{t,x}^{\theta q}}^\theta \| e^{it\Delta} P_{N_n} u_n\|_{L_{t,x}^\infty}^{1-\theta} \\ 
& \lesssim \| |\nabla|^{\frac{d}{2}-\frac{2}{p\theta}} P_{N_n}u_n\|_{L_x^2}^\theta  \| e^{it\Delta} P_{N_n} u_n\|_{L_{t,x}^\infty}^{1-\theta} \\
& \lesssim N_n^{-\frac{2(1-\theta)}{p}}\|u_n\|_{\dot H^{s_c}}^\theta  \| e^{it\Delta} P_{N_n} u_n\|_{L_{t,x}^\infty}^{1-\theta}.
\end{align*}
It follows that there exist $(t_n,x_n)\in\R^{1+d}$ such that
\begin{equation}\label{is-nontriv}
N_n^{-\frac2p}|(e^{it_n\Delta} P_{N_n}u_n)(x_n)| \gtrsim \beta:=A(\tfrac{\eps}{A})^{\frac{q}{(q-2)(1-\theta)}}. 
\end{equation}

By Banach--Alaoglu, there is a subsequence along which
\[
N_n^{-\frac2p}(e^{it_n\Delta}u_n)(\tfrac{x}{N_n}+x_n)
\]
converges weakly in $\dot H^{s_c}$.  Taking inner products with the test function $P_1\delta$ and using \eqref{is-nontriv} shows that this weak limit is non-trivial, having $\dot H^{s_c}$ norm is bounded below by $\beta$.

Finally, we observe that the sequence $\{N_n x_n\}$ is bounded.  Indeed, by radial Sobolev embedding, 
\[
|x_n|^{\frac{d-1}{2}}|(e^{it_n\Delta}P_{N_n}u_n)(x_n)| \lesssim \|P_{N_n}u_n\|_{L_x^2}^{\frac12}\|\nabla P_{N_n}u_n\|_{L_x^2}^{\frac12} \lesssim N_n^{\frac12-s_c}A,
\] 
which combined with \eqref{is-nontriv} yields 
$$
N_n |x_n| \lesssim [A\beta^{-1} ]^{\frac2{d-1}}.
$$
Thus, passing to a subsequence, we conclude that $N_nx_n$ converges to a fixed point in $\R^d$.  After an additional translation, we may thus assume $x_n\equiv 0$.
\end{proof}

As stated above, successive applications of Proposition~\ref{P:IS} yield the linear profile decomposition Proposition~\ref{P:LPD}. For a textbook treatment, see \cite{Visan-notes}.  

To prove Theorem~\ref{T:reduction}, we set 
\[
L(E) := \sup \norm{u}_{S(K)},
\]
where the supremum is taken over all compact time intervals $K$ and over all radial solutions $u\in C(K;\dot{H}^{s_c})$ satisfying $\sup_{t\in K} \norm{u(t)}_{\dot{H}^{s_c}} \le E$. By the results in  \cite{Ma3,KMMV}, we have the following characterization of the quantity $E_c$ introduced in \eqref{Ec}:
\[
E_c = \sup \{E  :\, L(E)<\infty\}.
\]
Hence, with the linear profile decomposition Proposition~\ref{P:LPD} and the stability theory of Section~\ref{S:LWP} in place, the existence of an almost periodic modulo symmetries solution attaining $E_c$ follows from well-known arguments; see, for example, \cite{Visan-notes}.  Finally, the rescaling and limiting arguments in \cite{KTV,TVZ} apply in our setting and allow us to reduce attention to the three scenarios recorded in Theorem~\ref{T:reduction}.

We conclude this section with a few useful properties of almost periodic solutions. 

\begin{lemma}[Uniform smallness]\label{L:unif-small} Let $u:I\times\R^d\to\C$ be a radial maximal-lifespan almost periodic solution to \eqref{nls}.  For any $\eps>0$, there exists $\delta=\delta(u,\eps)>0$ such that for any $t_0\in I$, we have
\[
 I_\delta(t_0):=[t_0-\delta N(t_0)^{-2}, t_0+\delta N(t_0)^{-2}] \subset I\qtq{and} \|u\|_{S(I_\delta(t_0))}<\eps.
\]
\end{lemma}

\begin{proof}  First, note that we may choose $\delta=\delta(u)$ small enough that $I_\delta(t_0)\subset I$ for any $t_0\in I$ (see e.g. \cite[Lemma~5.18]{KVClay}).

Let $\eps>0$.  We first claim that there exists $\delta=\delta(u,\eps)>0$ such that
\begin{equation}\label{delta1}
\| e^{i(t-t_0)\Delta} u(t_0)\|_{S(I_\delta(t_0))} < \eps \qtq{for any}t_0\in I. 
\end{equation} 
To see this, we first use almost periodicity to write
\[
u(t,x) = N(t)^{\frac{2}{p}}\psi(t)(N(t)x),
\]
where $\psi:I\to K$ for some precompact set $K\subset\hsc $.  A change of variables shows
\[
\| e^{i(t-t_0)\Delta}u(t_0)\|_{S(I_\delta(t_0))}= \| e^{it\Delta}\psi(t_0)\|_{S([-\delta,\delta])}. 
\]
Using Strichartz, monotone convergence, and the precompactness of $K$, we deduce that for $\delta=\delta(u,\eps)$ small enough,
\[
\sup_{t_0\in I} \|e^{it\Delta}\psi(t_0)\|_{S([-\delta,\delta])} < \eps.
\]
This gives \eqref{delta1}.

Next, we apply the stability theory with the approximate solution 
\[
\tilde u(t) = e^{i(t-t_0)\Delta}u(t_0)
\] 
on the interval $I_\delta(t_0)$.  Note that $\tilde u$ solves \eqref{nls} up to an error of
\[
\| |\tilde u|^p \tilde u \|_{N(I_\delta(t_0))} \lesssim \| \tilde u\|_{S(I_\delta(t_0))}^{p+1}\lesssim \eps^{p+1}.
\]
Furthermore, $\tilde u(t_0)=u(t_0)$.  Thus, for $\eps$ small we can apply the stability result Theorem~\ref{T:stab} to deduce that
\[
\|u-\tilde u\|_{S(I_\delta)}\lesssim \eps^{p+1},\qtq{whence} \|u\|_{S(I_\delta)}\lesssim \eps.
\]
The result follows. \end{proof}

Using that $N(t)=t^{-\frac12}$ for self-similar solutions and that $N(t)\leq 1$ in the remaining two scenarios, we deduce:

\begin{corollary}\label{C:unif-small} Let $u:I\times\R^d\to\C$ be a radial almost periodic solution as in Theorem~\ref{T:reduction}.  For any $\eps>0$, there exists $0<\delta=\delta(u,\eps)<1$ such that:
\begin{itemize}
\item If $u$ is self-similar, then 
\[
\sup_{T\in(0,\infty)} \| u\|_{S((1-\delta)T,(1+\delta)T)} < \eps.
\]
\item If $u$ is a cascade or soliton solution, then
\[
\| u\|_{S(I)} <\eps\qtq{for any interval such that} |I|< 2\delta. 
\] 
\end{itemize}
\end{corollary}

As a consequence of Lemma \ref{L:unif-small} and the frequency-localized Strichartz estimate Proposition \ref{P:NStrichartz}, we obtain

\begin{lemma}\label{stb}
Let $u:I\times\R^d\to\C$ be a radial almost periodic solution to \eqref{nls}  as in Theorem~\ref{T:reduction}.  Given $\eps>0$, let $\delta=\delta(u,\eps)>0$ be as in Lemma \ref{L:unif-small}.  Then, for any $t_0\in I$ and $(q,r,s,\alpha)\in (2,\infty)\times (2,\infty) \times \R \times \R$ satisfying \eqref{scalecond} and \eqref{alpharange}, 
\[
\|P_{M} u \|_{L^q_tL_x^{r;r\alpha}(I_\delta(t_0)\times \R^d)}\lesssim_{p,q,s,\alpha} M^{s-s_c} \eps.
\]
The operator $P_{M}$ can be replaced by $P_{>M}$ if $s\leq s_c$ and by $P_{<M}$ if $s\geq s_c$.
\end{lemma}

Finally we have the following reduced Duhamel formula for almost periodic solutions, which follows from the fact that $e^{-it\Delta}u(t)$ must converge weakly to zero in $\dot H^{s_c}$ at the endpoints of its maximal lifespan (cf.  \cite[Proposition~5.23]{KVClay}).

\begin{lemma}[Reduced Duhamel formula]\label{L:RDF} Let $u:I\times\R^d\to\C$ be a maximal-lifespan almost periodic solution to \eqref{nls}. Then
\begin{align*}
u(t)=i\lim_{T\to\sup I}\int_t^T e^{i(t-s)\Delta}(\mu |u|^p u)(s) \,ds = -i\lim_{T\to\inf I}\int_T^t e^{i(t-s)\Delta}(\mu |u|^p u)(s)\,ds,
\end{align*}
as weak limits in $\dot H^{s_c}$. 
\end{lemma}

\section{The self-similar scenario}\label{S:SS}

In this section we prove the following:

\begin{theorem}\label{T:no-ss} There are no self-similar solutions as in Theorem~\ref{T:reduction}.
\end{theorem}

It suffices to prove that any self-similar solution must belong to $L^2$.  Indeed, $L^2$ solutions to the mass-subcritical NLS are automatically global, while self-similar solutions blow up at $t=0$.  Throughout this section, we write $F(z)=\mu |z|^p z$. 

\begin{proof}[Proof of Theorem~\ref{T:no-ss}] Suppose toward a contradiction that $u$ is a self-similar solution in the sense of Theorem~\ref{T:reduction}.  We fix $\eps>0$ to be determined below (cf. Lemma~\ref{L:initialS}), and choose $\delta=\delta(u,\eps)$ as in Corollary~\ref{C:unif-small}.  For $T\in(0,\infty)$ we set 
\[
I_\delta(T):=(T,(1+\delta)T).
\]  
For $A>0$, we define
\begin{align}
\M(A) &:= \sup_{T\in(0,\infty)} \| \nsc u_{>AT^{-\frac12}}(T)\|_{L_x^2}, \\
\S(A) &:= \sup_{T\in(0,\infty)} \| u_{>AT^{-\frac12}}\|_{S(I_\delta(T))}, \\
\N(A) &:= \sup_{T\in(0,\infty)} \| P_{>AT^{-\frac12}} F(u) \|_{N(I_\delta(T))},
\end{align}
where we use the notation $S(I)$ and $N(I)$ introduced in Section~\ref{S:FS}.  By Strichartz,
\begin{align}\label{10am}
\S(A)\lesssim \M(A)+\N(A) \quad\text{uniformly in $A>0$},
\end{align}
while by Corollary~\ref{C:unif-small} and H\"older,
\begin{equation}\label{MSN-initial-small}
\M(A)\lesssim_u 1,\quad \S(A)\lesssim \eps,\qtq{and} \N(A)\lesssim \eps^{p+1} \quad\text{uniformly in $A>0$}.
\end{equation}

To prove $u\in L^2$, we will prove quantitative decay for $\M(A)$ as $A\to\infty$.   The first result towards this goal shows that decay of $\N$ implies decay of $\M$ and $\S$. 

\begin{lemma}[$\N$ controls $\M$ and $\S$]\label{NcontrolsM} For $\sigma>0$,
\[
\N(A)\lesssim A^{-\sigma}\implies \M(A) + \S(A)\lesssim A^{-\sigma}. 
\]
\end{lemma}

\begin{proof} Suppose $\N(A)\lesssim A^{-\sigma}$.  By the reduced Duhamel formula (Lemma~\ref{L:RDF}), weak lower-semicontinuity, and Strichartz, we have
\begin{align*}
\| \nsc u_{>AT^{-\frac12}}(T)\|_{L_x^2} & \leq \sum_{k=0}^\infty\biggl\|\nsc \int_{I_\delta((1+\delta)^kT)} e^{i(T-s)\Delta}P_{>AT^{-\frac12}}F(u(s))\,ds\biggr\|_{L_x^2} \\
&\lesssim \sum_{k=0}^\infty \| P_{>AT^{-\frac12}} F(u)\|_{N(I_\delta((1+\delta)^kT))} \\
& \lesssim \sum_{k=0}^\infty \N((1+\delta)^{\frac{k}{2}} A)  \lesssim \sum_{k=0}^\infty (1+\delta)^{-\frac{k\sigma}{2}} A^{-\sigma} \lesssim A^{-\sigma}. 
\end{align*}
Thus $\M(A)\lesssim A^{-\sigma}$.  Combining this with \eqref{10am} yields  $\S(A)\lesssim A^{-\sigma}$, as well.
\end{proof}

The next lemma concerns the estimation of $\N(A)$. 

\begin{lemma}[Control of $\N$]\label{controlN}
Fix $0<\beta \le 1$.  There exist $k \in (0,1)$ and $\tilde\sigma>0$ depending only on $p$, $q$, and $d$
such that
\begin{align}\label{SN1}
\N(A) \lesssim \eps^p\!\! \sum_{N\leq A^\beta}\! \tfrac{N}{A} \S(N) &+ A^{-\tilde\sigma \beta} [\M ( A^\beta) + \N( A^\beta)]\notag\\
&+ [\S( A^{k\beta})^p + \S( A^\beta)^p]\S( A^\beta),
\end{align}
uniformly in $A>0$.
\end{lemma}

\begin{proof} Fix $\alpha = k\beta$, where $k\in(0,1)$ will be determined below.  We decompose the nonlinearity as follows:
\begin{align}
F(u) & = F(u_{\leq A^\beta T^{-\frac12}}) + \text{\O}(|u_{\leq A^\alpha T^{-\frac12}}|^p u_{>A^\beta T^{-\frac12}}) \label{decomp1} \\
& \quad +{\O}(|u_{A^\alpha T^{-\frac12}<\cdot\leq A^\beta T^{-\frac12}}|^p u_{>A^\beta T^{-\frac12}}) + {\text{\O}}(|u_{>A^\beta T^{-\frac12}}|^{p+1}). \label{decomp2}
\end{align}

We firstly estimate the terms in \eqref{decomp2} via H\"older's inequality:
\[
\|P_{>AT^{-\frac12}}\eqref{decomp2}\|_{N(I_\delta(T))} \lesssim [\S( A^\alpha)^p + \S( A^\beta)^p]\S( A^\beta),
\]
which is acceptable. 

By weighted Bernstein (Lemma~\ref{L:WB}), H\"older, and Corollary~\ref{C:unif-small},
\begin{align*}
\| P_{>A T^{-\frac12}}F(u_{\leq A^\beta T^{-\frac12}}) \|_{N(I_\delta(T))} & \lesssim (AT^{-\frac12})^{-1} \| \text{\text{\O}}(u^p \nabla u_{\leq {A^\beta}T^{-\frac12}}) \|_{N(I_\delta(T))} \\
& \lesssim \eps^p \sum_{N\leq A^\beta} \tfrac{N}{A} \S(N),
\end{align*}
which is acceptable.  Note that the application of Lemma~\ref{L:WB} requires
\[
-d<\gamma<d(\tfrac{q}{p+1}-1),
\]
which follows from \eqref{q-conditions}. 

We turn to the second term in \eqref{decomp1}.  We fix a parameter $\theta=\theta(p,q)$ satisfying
\[
0<\theta<\min\{1,p,\tfrac{(q-2)p}2,\tfrac{2(q-p)}{2-q|s_c|}\}
\]
and define
\[
r = \tfrac{2q(p-\theta)}{2p-\theta(q-2)}\qtq{and} a =-\tfrac{p+\theta}{p-\theta}\tfrac{\gamma}{q}.
\]
Note that $r\in (2,\infty)$ since $\tfrac{2}{q}\in(\tfrac{1}{p+1},1)$. Recalling that $\gamma<0$, we also have $a>0$. Furthermore, using the upper bound on $\theta$, we find
\begin{equation}\label{ss-hardy}
\tfrac{d}{r}-a = \tfrac{p+\theta}{p-\theta}(\tfrac{2}{p}-\tfrac{2}{q})-\tfrac{\theta}{p-\theta}\tfrac{d}{2}>0. 
\end{equation}

Discarding the projection to high frequencies and taking all space-time norms over $I_\delta(T)\times\R^d$, we use H\"older's inequality to estimate
  \begin{align*}
\|& |u_{\leq A^\alpha T^{-\frac12}}|^p u_{>A^\beta T^{-\frac12}} \|_{N(I_\delta(T))} \\
 & \lesssim \| u_{\leq A^\alpha T^{-\frac12}} u_{>A^\beta T^{-\frac12}} \|_{L_{t,x}^2}^\theta \| |x|^{-a} u_{\leq A^{\alpha}T^{-\frac12}} \|_{L_{t,x}^r}^{p-\theta} \| u_{>A^{\beta}T^{-\frac12}} \|_{S(I_\delta(T))}^{1-\theta}. \end{align*}
We will now estimate each of these terms individually.  

By the bilinear estimate from Corollary~\ref{C:BS} and Corollary~\ref{C:unif-small}, we get
\begin{align*}
\| u_{\leq A^\alpha T^{-\frac12}} u_{>A^\beta T^{-\frac12}} \|_{L_{t,x}^2} \lesssim A^{\alpha(\frac{d-1}{2}+|s_c|)-\beta(\frac12-|s_c|)} T^{-\frac12(\frac{d}{2}-1+2|s_c|)}[\M(A^\beta)\!+\!\N(A^\beta)].
\end{align*}
On the other hand, by H\"older (in time), Hardy (cf. \eqref{ss-hardy}), and Bernstein,
\[
\| |x|^{-a}u_{\leq A^\alpha T^{-\frac12}} \|_{L_{t,x}^r} \lesssim A^{\alpha(\frac{d}{2}-\frac{d}{r}+a+|s_c|)} T^{-\frac12(\frac{d}{2}-\frac{d+2}{r}+|s_c|+a)} \|\nsc u\|_{L_t^\infty L_x^2}. 
\]
Finally, an application of Strichartz yields
\[
\| u_{>A^{\beta}T^{-\frac12}} \|_{S(I_\delta(T))} \lesssim \M(A^\beta)+\N(A^\beta). 
\]

We now combine these three estimates and collect the total powers of $T^{-\frac12}$ and $A$.  For the power of $T^{-\frac12}$, we have
\[
\theta(\tfrac{d}{2}-1+2|s_c|)+(p-\theta)(\tfrac{d}{2}-\tfrac{d+2}{r}+|s_c|+a) = 0.
\]
Next, using the scaling relations (e.g. \eqref{ss-hardy} above) and recalling $\alpha=k\beta$, we write the power of $A$ as
\begin{align*}
&\theta[\alpha(\tfrac{d-1}{2}+|s_c|)-\beta(\tfrac12-|s_c|)]+(p-\theta)\alpha(\tfrac{d}{2}-\tfrac{d}{r}+a+|s_c|) \\
& = \alpha \bigl[\tfrac{2p}{q}+ \theta(\tfrac2q -\tfrac12 + |s_c|)\bigr] - \theta\beta(\tfrac12-|s_c|)  \\
&= \beta\bigl\{k\bigl[\tfrac{2p}{q} + \theta(\tfrac2q -\tfrac12 + |s_c|) \bigr]-\theta(\tfrac12-|s_c|)\bigr\}.
\end{align*}
This may be written in the form $-\tilde\sigma\beta$ for $\tilde\sigma>0$ provided
\begin{align*}
0<k<\tfrac{\theta(\frac12-|s_c|)}{\frac2q(p+\theta) + \theta(\frac12-|s_c|)}<1.
\end{align*}
This completes the estimation of the second term in \eqref{decomp1} and so the proof of Lemma~\ref{controlN}. \end{proof}

Combining Lemmas~\ref{NcontrolsM} and \ref{controlN} leads to the following: 

\begin{corollary}\label{MSNcontrolsN}
Let $0<b<1$.  There exists $0<\sigma_1<1-b$ such that the following holds:  for any $0<\sigma<b$, 
\begin{equation}\label{E:MSNcontrolsN}
\M(A) + \S(A)+ \N(A) \lesssim A^{-\sigma}	\implies \M(A)+\S(A)+ \N(A) \lesssim A^{-\sigma -\sigma_1}.
\end{equation}
\end{corollary}

\begin{proof}
Suppose the left-hand side of \eqref{E:MSNcontrolsN} holds for some $0<\sigma<b<1$.  We apply Lemma \ref{controlN} to obtain
\[
	\N (A)
\lesssim A^{-\beta\sigma-(1-\beta)} + A^{-\beta(\sigma+\tilde\sigma)} + A^{-\beta\sigma(1+kp)}.
\]
Choosing $0<\beta=\beta(k,p,\sigma,\tilde\sigma)<1$ sufficiently close to $1$ yields the desired estimate for $\N(A)$.  Combining this with Lemma~\ref{NcontrolsM} we obtain the same bounds for $\M(A)$ and $\S(A)$.  
\end{proof}

With Corollary~\ref{MSNcontrolsN} in place, it remains to obtain some initial quantitative decay.

\begin{lemma}[Initial estimate for $\S$]\label{L:initialS} There exists $0<\sigma_0<1$ so that
\begin{equation}\label{initialS}
\S(A)\leq C(u, \delta)  A^{-\sigma_0} + C(u) \eps^p \sum_{N\leq \frac{A}{2}} \tfrac{N}{A}\S(N) \qtq{uniformly in $A>0$.}
\end{equation}
In particular, choosing $\eps$ sufficiently small, we get
\begin{equation}\label{quantS}
\S(A) \lesssim  A^{-\sigma_0} \qtq{uniformly in $A>0$.}
\end{equation}
\end{lemma}

\begin{proof} We first prove \eqref{initialS}.  Writing the Duhamel formula for $u$ beginning at $t=\frac{T}{1+\delta}$ and applying Strichartz, we get
\begin{align}
\|u_{>AT^{-\frac12}}\|_{S(I_\delta(T))} & \lesssim \| P_{>AT^{-\frac12}}e^{i(t-\frac{T}{1+\delta})\Delta}u(\tfrac{T}{1+\delta}) \|_{S(I_\delta(T))} \label{initialS1}\\ 
& \quad + \|P_{>AT^{-\frac12}} F(u)\|_{N( I_\delta(\frac{T}{1+\delta})\cup I_\delta(T))}.\label{initialS2}
\end{align}

By Lemma~\ref{controlN} (with $\beta=1$) and \eqref{MSN-initial-small}, we have
\[
\eqref{initialS2} \lesssim \N(A)+\N(A(1+\delta)^{-\frac12})\lesssim \N(\tfrac{A}{2}) \lesssim \text{RHS}\eqref{initialS}. 
\]

We will treat \eqref{initialS1} by estimating at a fixed frequency $B>A$ and summing.  Fix $\theta\in(0,1)$ (close to one) to be determined below and define
\[
a=\tfrac{q\theta}{\theta(q-1)+p(1-\theta)-(q-2)}, \quad b=\tfrac{q}{q-(p+1)},\quad \alpha= \tfrac{\gamma[p+2-\theta(p+1)]}{q\theta}.  
\]
Note that $a>2$ for $\theta$ close to $1$.  By H\"older's inequality, we get 
\begin{align}
\| e^{i[t-\frac{T}{1+\delta}]\Delta} P_{BT^{-\frac12}} u(\tfrac{T}{1+\delta})  \|_{S(I_\delta(T))} &\lesssim \| |x|^{\alpha} e^{i[t-\frac{T}{1+\delta}]\Delta}P_{BT^{-\frac12}} u(\tfrac{T}{1+\delta})\|_{L_{t,x}^a}^{\theta}  \label{ss-q1} \\
&\quad \times  \| |x|^{-\frac{\gamma}{q}(p+1)}e^{i[t-\frac{T}{1+\delta}]\Delta} u(\tfrac{T}{1+\delta}) \|_{L_{t,x}^b}^{1-\theta},\label{ss-q2}
\end{align}
where all space-time norms are over $I_\delta(T)\times\R^d$. 

To estimate \eqref{ss-q1} we use Theorem~\ref{T:RS}. This requires $\alpha>-\tfrac{d}{a}$, which we can guarantee by choosing $\theta$ close enough to $1$.  Indeed, when $\theta=1$, we have $\alpha=\tfrac{\gamma}{q}$ and $a=q$, in which case the desired bound follows from $p<q$.  Setting $s=\tfrac{d}{2}-\tfrac{d+2}{a}-\alpha$, we then estimate
\begin{align*}
\eqref{ss-q1} & \lesssim \| |\nabla|^{s}P_{BT^{-\frac12}} u(\tfrac{T}{1+\delta})\|_{L_x^2}^\theta \lesssim (BT^{-\frac12})^{\theta(s-s_c)} \lesssim (BT^{-\frac12})^{-2(1-\theta)|s_c|},
\end{align*}

We estimate \eqref{ss-q2} using the weighted dispersive estimate (Lemma~\ref{L:WDE}). This requires
\[
\tfrac{d-1}{2}[1-\tfrac{2}{b}]+\tfrac{\gamma}{q}(p+1)=\tfrac52-s_c-\tfrac3q(p+1)\geq 0,
\]
which follows from the final upper bound in \eqref{q-conditions}.  Using the reduced Duhamel formula (Lemma~\ref{L:RDF}) and Lemma~\ref{L:WDE}, we get
\begin{align*}
\| &|x|^{-\frac{\gamma}{q}(p+1)} e^{i[t-\frac{T}{1+\delta}]\Delta} u(\tfrac{T}{1+\delta})\|_{L_{t,x}^b} \\
& \lesssim T^{\frac{1}{b}} \biggl\| \int_0^{\frac{T}{1+\delta}} |t-s|^{-(\frac{d}{2}-\frac{d}{b})-\frac{\gamma}{q}(p+1)} \| |x|^{\frac{\gamma}{q}(p+1)} F(u(s)) \|_{L_x^{\frac{q}{p+1}}} \, ds\biggr\|_{L_t^\infty(I_\delta(T))} \\
& \lesssim_\delta T^{\frac{1}{b}-(\frac{d}{2}-\frac{d}{b})-\frac{\gamma}{q}(p+1)} \|\,|x|^{\frac{\gamma}{q}(p+1)}F(u)\|_{L_t^1 L_x^{\frac{q}{p+1}}((0,\frac{T}{1+\delta})\times\R^d)} \\
& \lesssim_\delta T^{\frac{1}{b}-(\frac{d}{2}-\frac{d}{b})-\frac{\gamma}{q}(p+1)}\sum_{0< \tau \leq \frac{T}{2(1+\delta)}} \tau^{\frac{1}{b}} \|u\|_{S([\tau,2\tau])}^{p+1}  \lesssim_\delta T^{-\frac{d}{2} + \frac{d+2}{b} - \frac{\gamma}{q}(p+1)} \lesssim_\delta T^{-|s_c|}. 
\end{align*}

Summing over $B>A$ we obtain
\[
\|  e^{i[t-\frac{T}{1+\delta}]\Delta} P_{>AT^{-\frac12}} u(\tfrac{T}{1+\delta})   \|_{S(I_\delta(T))} \lesssim_\delta A^{-2(1-\theta)|s_c|}. 
\]  
This completes the proof of \eqref{initialS}.

We turn to \eqref{quantS}.  We first  rewrite \eqref{initialS} as
\begin{equation}\label{initialSagain}
\mathcal S(A) \leq C_1(u,\delta) A^{-\sigma_0} +C(u) \eps^p \sum_{N\leq\frac{A}{2}} \tfrac{N}{A} \S(N) \qtq{uniformly in $A>0$.}
\end{equation}
The goal is to prove
\begin{equation}\label{quantS2}
\S(A) \leq C_2(u,\delta) A^{-\sigma_0} \qtq{uniformly in $A>0$.}
\end{equation}

By \eqref{MSN-initial-small}, \eqref{quantS2} clearly holds for $A\leq 1$ and $C_2$ sufficiently large.  Now suppose \eqref{quantS2} holds up to $\frac{A}{2}$.  Then using \eqref{initialSagain} and the inductive hypothesis, we deduce
\[
\S(A) \leq \bigl[C_1(u,\delta)+C_2(u,\delta)C(u) \eps^p\bigr]A^{-\sigma_0}. 
\]
Taking $C_2\geq 2C_1$ and $\eps$ small enough so that $\eps^p C(u)\leq 1$, we derive \eqref{quantS2} at level $A$.
 \end{proof}

We are now in a position to complete the proof of Theorem \ref{T:no-ss}.  We first claim that 
\begin{align}\label{1135}
\M(A)+\S(A)+\N(A) \lesssim A^{-\sigma_0} \quad\text{uniformly for $A>0$}
\end{align}
and some $\sigma_0>0$.   In view of \eqref{MSN-initial-small}, we only need to verify this claim for $A$ large.  By Strichartz combined with Lemmas~\ref{L:initialS} and \ref{controlN} (with $\beta=1$), we find
\[
\M(A) + \N(A)\lesssim A^{-\sigma_0} + A^{-\tilde\sigma}[\M(A)+\N(A)].
\] 
Taking $A$ sufficiently large, we deduce \eqref{1135}. 

We now choose $2|s_c|<b<1$ and apply Lemma~\ref{MSNcontrolsN} finitely many times to deduce
\[
\M(A) \lesssim A^{-2|s_c|}.
\]
Thus, for any $t\in(0,\infty)$, we have
\begin{align*}
\| u_{>At^{-\frac12}}(t)\|_{L_x^2} &\lesssim \sum_{B>A} (Bt^{-\frac12})^{|s_c|} \|\nsc u_{Bt^{-\frac12}}(t)\|_{L_x^2} \lesssim A^{-|s_c|}t^{-\frac{|s_c|}{2}},
\end{align*}
while by Bernstein,
\begin{align*}
\| u_{\leq At^{-\frac12}}(t)\|_{L_x^2}& \lesssim A^{|s_c|}t^{-\frac{|s_c|}{2}}\| \nsc u\|_{L_x^2}\lesssim A^{|s_c|}t^{-\frac{|s_c|}{2}}.
\end{align*}
We conclude that $u(t)\in L^2$, which implies that $u$ is global.  This contradicts the fact that the self-similar solution blows up at $t=0$ and completes the proof of Theorem~\ref{T:no-ss}.
\end{proof}

\section{Additional regularity}\label{S:AR}

In this section, we establish additional regularity for the cascade and soliton scenarios described in Theorem~\ref{T:reduction}.  It is here that we encounter the technical restriction $p>p_0(d)$, where $p_0(d)$ is defined in \eqref{range}; see the proof of Lemma~\ref{L:MC}. 

We will employ an in/out decomposition for radial functions, as introduced in \cite{KTV, KVZ}. Briefly, writing a radial function as $f=f(r)$, one defines the projection onto outgoing/incoming spherical waves by
\begin{align*}
[P^+ f](r) &= \tfrac12 f(r) + \tfrac{i}{\pi} \int_0^\infty \tfrac{r^{2-d}\rho^{d-1}}{r^2-\rho^2}f(\rho)\,d\rho, \\
[P^- f](r) & = \tfrac12 f(r) -\tfrac{i}{\pi} \int_0^\infty \tfrac{r^{2-d}\rho^{d-1}}{r^2-\rho^2}f(\rho)\,d\rho.
\end{align*}
We write $P_M^\pm = P^{\pm} P_M$.  We import a few results from \cite[Section~4]{KVZ}. 

\begin{lemma}[Kernel estimates]\label{L:KE} Fix $M\in 2^\Z$.  For $|x|\gtrsim M^{-1}$, $t\gtrsim M^{-2}$, and $m\geq 0$,
\[
\bigl| [P_M^\pm e^{\mp it\Delta}](x,y)\bigr| \lesssim_m \begin{cases} \bigl( |x|\,|y|\bigr)^{-\frac{d-1}{2}} |t|^{-\frac12} &: |y|-|x|\sim Mt \\ \\
\frac{M^d}{(M|x|)^{\frac{d-1}{2}}\langle M|y|\rangle^{\frac{d-1}{2}}} \langle M^2 t + M|x|\!-\!M|y|\rangle^{-m} &: \text{otherwise.}\end{cases}
\]
For $|x|\gtrsim M^{-1}$, $|t|\lesssim M^{-2}$, and $m\geq 0$,
\[
\bigl| [P_M^\pm e^{\mp it\Delta}](x,y)\bigr|\lesssim_m \frac{M^d}{(M|x|)^{\frac{d-1}{2}}\langle M|y|\rangle^{\frac{d-1}{2}}} \langle M|x|-M|y|\rangle^{-m}.
\] 
\end{lemma}

\begin{lemma} The operator $P^++P^-$ is the identity on radial functions in $L^2(\R^d)$.  Moreover, for a spherically-symmetric function $f\in L^2(\R^d)$,
\[
\| P^{\pm} P_{>M} f\|_{L_x^2(|x|\geq \frac{1}{100}M^{-1})} \lesssim \|f\|_{L_x^2(\R^d)}, 
\]
uniformly for $M\in 2^\Z$.
\end{lemma}

We will also need the following kernel estimates for the frequency-localized free propagator:
\begin{lemma}[Kernel estimates]\label{L:KE2} For any $m\geq 0$ and $M\in 2^{\mathbb{Z}}$, 
\[
|P_M e^{it\Delta}(x,y)| \lesssim_m \begin{cases}
M^d\langle M|x-y|\rangle^{-m} &: |t|\leq M^{-2} \\
|t|^{-\frac{d}{2}} &: |t|>M^{-2}\qtq{and} |x-y|\sim M|t| \\
\tfrac{M^d}{|M^2 t|^m\langle M|x-y|\rangle^m} &: \text{otherwise}.
\end{cases}
\]
\end{lemma}

As usual, we denote $F(z)=\mu |z|^p z$. 

\begin{theorem}[Additional regularity]\label{T:regularity} Suppose $u$ is a cascade or soliton solution as in Theorem~\ref{T:reduction}. Then 
$u\in L_t^\infty H_x^{\frac12}(\R\times\R^d)$.
\end{theorem}

\begin{proof} Suppose $u$ is a cascade or soliton solution as in Theorem~\ref{T:reduction}. For $N\in 2^\Z$, we define 
\begin{align*}
\M(N) := \| u_{\geq N}\|_{L_t^\infty \dot H_x^{s_c}(\R\times\R^d)}\sim \Bigl\|\sqrt{\sum_{M\geq N} M^{2s_c} \|u_M(t)\|_{L_x^2}^2}\Bigr\|_{L_t^\infty}. 
\end{align*}
Note that $\M(N)\lesssim_u1$ uniformly in $N$, 
\[
\lim_{N\to\infty} \M(N) = 0 \qtq{and} \M(N)\lesssim \Bigl\|\sum_{M\geq N} M^{s_c} \|u_M(t)\|_{L_x^2}\Bigr\|_{L_t^\infty}.
\]

We will establish the following recurrence relation for $\M(N)$:

\begin{lemma}\label{lem:addreg} There exists $s>\tfrac12-s_c$ so that 
\[
\M(N)\lesssim_u N^{-s} + \sum_{M\leq \eta N}(\tfrac{M}{N})^s\M(M)
\]
for sufficiently small $\eta$ and sufficiently large $N$.
\end{lemma}

Combining Lemma~\ref{lem:addreg} with the Gronwall-type inequality of Lemma~\ref{L:AG}, we derive that there exists $s>\tfrac12-s_c$ such that
\begin{equation}\label{eq:addreg}
\M(N)\lesssim_u N^{-s}\quad\text{uniformly for $N\geq 1$}.
\end{equation}
Theorem~\ref{T:regularity} follows easily from this, since then
\begin{align*}
\|u\|_{L_t^\infty \dot H_x^{\frac12}} & \lesssim \|u_{\leq 1}\|_{L_t^\infty \dot H_x^{s_c}}+ \sum_{N\geq 1} N^{\frac12-s_c}\M(N)\lesssim_u 1.
\end{align*}

By time-translation symmetry, to prove Lemma~\ref{lem:addreg} it suffices to show that
\begin{equation}\label{ar-ets}
\sqrt{\sum_{M\geq N} M^{2s_c}\|u_M (0)\|_{L^2_x}^2} \lesssim_u N^{-s}+\sum_{M\leq\eta N}(\tfrac{M}{N})^s \M(M)
\end{equation}
for some $s>\tfrac12-s_c$ and for sufficiently small $\eta$ and sufficiently large $N$.

Let $\chi$ be the characteristic function of $[1,\infty)$ and let $\chi_N(x) = \chi(N|x|)$.  Using the in/out decomposition and the reduced Duhamel formula (Lemma~\ref{L:RDF}), we decompose
\begin{align}
\chi_N u_M(0)&= i \!\int_0^\delta \!\!\chi_N P^+_M e^{-it\Delta}  F(u(t))\, dt - i\! \int^0_{-\delta} \chi_N P^-_M e^{-it\Delta} F(u(t)) \,dt \label{eq:areg1}\\
&\quad+i\lim_{T\to\infty}   \int_\delta^T \chi_N P^+_M e^{-it\Delta}  F(u(t)) \,dt \label{eq:areg2}\\
&\quad- i\lim_{T\to\infty}  \int^{-\delta}_{-T} \chi_N P^-_M e^{-it\Delta}  F(u(t)) \,dt\label{eq:areg3}
\end{align}
and
\begin{align}
(1-\chi_N) u_M(0)&=i\int_0^\delta (1-\chi_N) e^{-it\Delta} P_M F(u(t)) \,dt \label{eq:areg4}\\
&\quad+ i\lim_{T\to\infty}\int^T_{\delta} (1-\chi_N) P_M e^{-it\Delta} F(u(t)) \,dt, \label{eq:areg5}
\end{align}
where $\delta>0$ will be determined below and the limits are in the weak $\dot H^{s_c}_x$ topology.

We first estimate the contribution of \eqref{eq:areg1} and \eqref{eq:areg4}. 

\begin{lemma}[Short-time contribution]\label{L:ST} Let $s\in(0,1)$. For $\eta>0$ sufficiently small,  there exists $\delta>0$ so that
\[
\sqrt{\sum_{M\geq N} M^{2s_c}\biggl\| \int_0^\delta e^{-it\Delta}P_M F(u(t))\,dt\biggr\|_{L^2_x}^2} \lesssim N^{-s}+\tfrac{1}{100}\sum_{M\leq \eta N}(\tfrac{M}{N})^s \M(M)
\]
for $N=N(u,s,\eta)$ large enough. Similar estimates hold over $[-\delta,0]$ and after multiplication by $\chi_N P^{\pm}$. 
\end{lemma}

\begin{proof} Fix $0<s<1$ and let $\eta>0$ to be determined below.  By Strichartz, it suffices to show
\[
\N(N) := \|P_{\geq N}F(u)\|_{N(0,\delta)} \lesssim N^{-s}+\sum_{M\leq \eta N}(\tfrac{M}{N})^s \M(M). 
\]
To this end, we first use Corollary~\ref{C:unif-small} to find $\delta=\delta(u,\eta)$ small enough that
\begin{equation}\label{addreg-small}
\|u\|_{S(0,\delta)}^p <\eta. 
\end{equation}
We then write
\[
P_{\geq N}F(u) = P_{\geq N} F(u_{\leq \eta N})+P_{\geq N} \text{\O}(u_{>\eta N}u^p). 
\]
By the weighted Bernstein inequality (Lemma~\ref{L:WB}), H\"older, and Strichartz,
\begin{align*}
\| P_{\geq N }F(u_{\leq \eta N}) \|_{N(0,\delta)} & \lesssim N^{-1} \| \nabla u_{\leq \eta N}\|_{S(0,\delta)} \|u_{\leq \eta N}\|_{S(0,\delta)}^p\\
& \lesssim \eta \sum_{M\leq \eta N} \tfrac{M}{N} \|u_M\|_{S(0,\delta)}\\
&\lesssim \sum_{M\leq \eta N} \tfrac{M}{N} [\M(M)+\N(M)].
\end{align*}
Next, by H\"older and Strichartz, 
\begin{align*}
\| u_{>\eta N}u^p \|_{N(0,\delta)}  & \lesssim \|u\|_{S(0,\delta)}^p[\M(\eta N)+\N(\eta N)]  \lesssim \eta[\M(\eta N)+\N(\eta N)].
\end{align*}

Collecting these estimates, we get
\[
\N(N) \lesssim \sum_{M\leq \eta N} \tfrac{M}{N}[\M(M)+\N(M)],
\]
which, combined with the Gronwall-type inequality of Lemma~\ref{L:AG}, yields the claim.
\end{proof} 

We turn to the estimation of the integrals over $|t|\geq\delta$, i.e. \eqref{eq:areg2}, \eqref{eq:areg3}, and \eqref{eq:areg5}.  We treat separately the main contribution coming from $|y|\gtrsim M|t|$ and that from the tail $|y|\ll M|t|$.  Let $\chi_{k,M}$ be the characteristic function of the set
\[
\{(t,y):2^k\delta\leq|t|\leq2^{k+1}\delta,\quad |y|\gtrsim M|t|\}. 
\]

\begin{lemma}[Main contribution]\label{L:MC} Let $A\in\{\chi_N P_M^\pm, (1-\chi_N)P_M\}$ and let $\delta$ be as in Lemma~\ref{L:ST}.  There exists $s>\tfrac12(\frac12-s_c)$ such that
\[
\sqrt{\sum_{M\geq N}M^{2s_c} \biggl\| \int_\delta^\infty Ae^{-it\Delta}\sum_k \chi_{k,M}F(u(t))\,dt\biggr\|_{L_x^2}^2} \lesssim_u (\delta N^2)^{-s},
\]
uniformly for $N>0$.  An analogous estimate holds on $(-\infty,-\delta)$. 
\end{lemma}

\begin{proof} It suffices to show
\[
\sum_{M\geq N}\sum_{k=0}^\infty M^{s_c} \biggl\| \int_\delta^\infty Ae^{-it\Delta}\chi_{k,M}F(u(t))\,dt\biggr\|_{L_x^2} \lesssim_u (\delta N^2)^{-s},
\]
uniformly for $N>0$.

Writing
\[
F(u-v)=(u-v)\int_0^1 F_z(v+\theta(u-v))\,d\theta + \overline{(u-v)}\int_0^1 F_{\bar z}(v+\theta(u-v))\,d\theta,
\]
we can decompose
\begin{equation}\label{bt}
\begin{aligned}
F(u) & = F(u_{\leq M}) + \sum_{L\geq 2M} [F(u_{\leq L})-F(u_{\leq \frac{L}{2}})]  \\
& =  F(u_{\leq M}) + \sum_{L\geq 2M} \text{\O}\bigl[u_L(u_{\leq L})^p\bigr]. 
\end{aligned}
\end{equation}

We begin by estimating the contribution of $F(u_{\leq M})$, namely,
\begin{equation}\label{eq:mc1}
\sum_{M\geq N}\sum_{k=0}^\infty M^{s_c}\biggl\| \int_\delta^\infty A e^{-it\Delta}\chi_{k,M} F(u_{\leq M})\,dt\biggr\|_{L_x^2}. 
\end{equation}
We remark that $A$ is bounded on $L^2$ (uniformly in $M\geq N$), and that {
by the radial Strichatrz estimate (cf. Proposition \ref{P:rad1} (i)) and Bernstein, 
\[
\| |x|^{\frac{d-1}{2}} u_{\leq M} \|_{L^4_tL^\infty_x(\R\times\R^d)} \lesssim_u M^{-s_c}
\langle 2^k \delta \rangle^{\frac14}. 
\]
Then by Strichartz, H\"older, and Bernstein,
\begin{align*}
\eqref{eq:mc1} & \lesssim \sum_{M\geq N}\sum_{k=0}^\infty M^{s_c} \|\chi_{k,M}F(u_{\leq M})\|_{L_t^1 L_x^2} \\
& \lesssim \sum_{M\geq N}\sum_{k=0}^\infty  (2^k\delta)^{1-\frac{p}4} M^{s_c}\|u_{\leq M}\|_{L_t^\infty L_x^2}(2^kM\delta)^{-\frac{p(d-1)}{2}}\| |x|^{\frac{d-1}{2}} u_{\leq M}\|_{L^4_t L_{x}^\infty}^p \\
& \lesssim_u \sum_{M\geq N}\sum_{k=0}^\infty (2^k \delta)^{1-\frac{p(2d-1)}{4}}\langle 2^k \delta \rangle^{\frac{p}4} M^{-\frac{p(d-1)}{2}-ps_c}
\lesssim_u (\delta N^2)^{-[\frac{p(2d-1)}{4}-1]}, 
\end{align*}
where we used the fact that $p>\tfrac{2}{d-1}$ to sum.  Note that
\[
\tfrac{p(2d-1)}{4}-1>\tfrac12(\tfrac12-s_c),
\]
provided
\begin{equation}\label{1207p1}
p>p_1(d)=\tfrac{5-d+\sqrt{d^2+22d+9}}{2(2d-1)}.
\end{equation}
We will collect all such constraints at the end of the lemma.}

We turn to
\begin{equation}\label{eq:mc2}
\sum_{M \geq N}\sum_{L\geq 2M}\sum_{k=0}^\infty M^{s_c} \biggl\| \int_\delta^\infty A e^{-it\Delta}\chi_{k,M}u_L(u_{\leq L})^p \,dt\biggr\|_{L_x^2}. 
\end{equation}

To estimate this term will require a careful choice of exponents and the interpolated Strichartz estimates in Theorem~\ref{T:RS}.  We introduce three sets of parameters, $(r_j,\alpha_j,s_j)$ that will satisfy the scaling relations
\begin{equation}\label{1207scaling}
\tfrac{d}{2}-s_j=\tfrac{d+2}{r_j}+\alpha_j,\quad j\in\{0,1,2\}.  
\end{equation}
We will take
\[
r_1=r_2,\quad s_1=s_c-,\quad s_2=s_c.
\]
For $r\in[2,6]$, the minimum appearing in \eqref{alpharange} of Theorem~\ref{T:RS} (restricted to the diagonal $q=r$) is given by
\[
\tfrac{2d-1}{4}-\tfrac{2d+1}{2r}. 
\]
For $\frac{4}{d+1}<p<\frac{4}{d}$ and
$$p>p_2(d):= \tfrac{2-d+\sqrt{d^2+12d-4}}{2d-1}$$
we may choose
\begin{equation}\label{r1r2}
\max\bigl\{2,\tfrac65(p+1)\bigr\}<r_1=r_2<\min\bigl\{\tfrac{6p}{8-p(2d-1)}, 2(p+1),6\bigr\}.
\end{equation}
This constraint guarantees that we may apply Strichartz (Theorem~\ref{T:RS}) with the weights
\begin{equation}\label{a1a2}
\alpha_1 = \tfrac2p-\tfrac{d+2}{r_1}+\qtq{and} \alpha_2=\tfrac2p-\tfrac{d+2}{r_2}=\tfrac2p-\tfrac{d+2}{r_1}. 
\end{equation}
It also guarantees that if we define
\begin{equation}\label{defr0}
r_0=\tfrac{r_1}{r_1-(p+1)}\qtq{so that} \tfrac{1}{r_0}+\tfrac{p+1}{r_1}=1,
\end{equation}
then $r_0\in(2,6)$.  Finally, we choose
\[
\alpha_0 = \tfrac{2d-1}{4}-\tfrac{2d+1}{2r_0}-,
\]
which also determines $s_0$ via \eqref{1207scaling}.  We now claim that for $r_1$ large enough, we have 
\begin{equation}\label{tosum1207}
\alpha_0+\alpha_1+p\alpha_2 > 1-\tfrac{1}{r_0},
\end{equation}
which will be used in order to sum below.  Indeed, using \eqref{a1a2} and \eqref{defr0}, we see that \eqref{tosum1207} is implied by
\begin{equation}\label{r1ub}
r_1>\tfrac{10p(p+1)}{8-p(2d-5)}.
\end{equation}
For \eqref{r1ub} to be compatible with \eqref{r1r2} requires
\[
p>
{
p_0(d)} := \begin{cases} \tfrac{15-2d+\sqrt{4d^2+100d+145}}{5(2d-1)} & 3 \leq d\leq 8 \\ \tfrac{4}{d+1} & d\geq 9.\end{cases}
\]

Having chosen exponents as above, we use Strichartz, H\"older, and Lemma~\ref{stb} to estimate
\begin{align*}
\eqref{eq:mc2} &\lesssim \sum_{M\geq N}\sum_{L\geq 2M}\sum_{k=0}^\infty M^{s_c+s_0} \| |x|^{-\alpha_0} \chi_{k,M} u_L (u_{\leq L})^p \|_{L_{t,x}^{r_0'}} \\
& \lesssim \sum_{M\geq N}\sum_{L\geq 2M}\sum_{k=0}^\infty M^{s_c+s_0}(2^k M\delta)^{-(\alpha_0+\alpha_1+p\alpha_2)} \| |x|^{\alpha_1}u_L\|_{L_{t,x}^{r_1}}\||x|^{\alpha_2}u_{\leq L}\|_{L_{t,x}^{r_2}}^p \\
& \lesssim_u \sum_{M\geq N}\sum_{L\geq 2M}\sum_{k=0}^\infty M^{s_c+s_0-(\alpha_0+\alpha_1+p\alpha_2)} (2^k\delta)^{-(\alpha_0+\alpha_1+p\alpha_2)}\langle 2^k\delta\rangle^{1-\frac{1}{r_0}}L^{s_1-s_c} \\
& \lesssim_u N^{s_1+s_0-(\alpha_0+\alpha_1+p\alpha_2)}\delta^{-(\alpha_0+\alpha_1+p\alpha_2)} \lesssim_u (N^2\delta)^{-(\alpha_0+\alpha_1+p\alpha_2)},
\end{align*}
where in the last line we use the scaling relations.  Finally, note that
\[
\alpha_0+\alpha_1+p\alpha_2 > \tfrac{p+1}{r_1} > \tfrac12(\tfrac12-s_c),
\]
provided
\[
r_1<\tfrac{4p(p+1)}{4-p(d-1)}.
\]
That this is compatible with \eqref{r1ub} follows from $p>\tfrac{4}{d+5}$. 

Collecting our estimates for \eqref{eq:mc1} and \eqref{eq:mc2}, we conclude that the desired estimate holds provided
\[
p>
{
\max\bigl\{p_0(d),p_1(d),p_2(d)\bigr\}
}
 =p_0(d).
\]
This completes the proof of Lemma~\ref{L:MC}. \end{proof}

We turn to the tail and denote by $\tilde\chi_{k,M}$ the characteristic function of the set
\[
\{(t,y): 2^k\delta\leq |t|\leq 2^{k+1}\delta,\quad |y|\ll M|t|\}. 
\]

\begin{lemma}[Tail estimate]\label{L:TE} For $N\gg \delta^{-\frac12}$, we have
\begin{align}
\sum_{M\geq N}\sum_{k=0}^\infty M^{s_c}\biggl\| \int_\delta^\infty [1-\chi_N] P_M e^{-it\Delta}\tilde\chi_{k,M} F(u)\,dt \biggr\|_{L_x^2} & \lesssim_u (\delta N^2)^{-50d}, \label{tail1}\\
\sum_{M\geq N}\sum_{k=0}^\infty M^{s_c}\biggl\| \int_\delta^\infty \chi_N P_M^\pm e^{-it\Delta}\tilde\chi_{k,M} F(u)\,dt \biggr\|_{L_x^2} & \lesssim_u (\delta N^2)^{-50d}.\label{tail2}
\end{align}
\end{lemma}

\begin{proof} We begin by combining Lemma~\ref{L:KE} and Lemma~\ref{L:KE2} to deduce the following kernel estimates: for $N\gg \delta^{-\frac12}$, 
\begin{align}
\bigl|[1-\chi_N(x)]P_M e^{-it\Delta}(x,y) \tilde \chi_{k,M}(y) \bigr| &\lesssim \frac{K_M(x,y)}{(M^2|t|)^{100d}}, \label{ke1}\\
|\chi_N(x) P_M^\pm e^{-it\Delta}(x,y) \tilde \chi_{k,M}(y)| &\lesssim \frac{K_M(x,y)}{(M^2|t|)^{100d} (M|x|)^{\frac{d-1}{2}}\langle M|y|\rangle^{\frac{d-1}{2}}},\label{ke2}
\end{align}
where
\[
K_M(x,y)=\frac{M^d}{\langle M(x-y)\rangle^{100d}} + \frac{M^d}{\langle M|x|-M|y|\rangle^{100d}}. 
\]
Indeed, for \eqref{ke1} we note that for $|t|\geq \delta \gg N^{-2}$, $|x|\leq N^{-1}$, and $|y|\ll M|t|$, we have $|y-x|\ll M|t|$. For \eqref{ke2}, we note that $|y|-|x|\ll M|t|$ under the given constraints; we also use
\[
\langle M^2|t| + M|x|-M|y|\rangle^{-200d}\lesssim (M^2|t|)^{-100d}\langle M|x|-M|y|\rangle^{-100d}. 
\]

Note that by Schur's test, we have
\begin{equation}\label{schur}
\|K_M\|_{L_x^{r'}\to L_x^{r}} \lesssim M^{d(1-\frac{2}{r})},\quad 2\leq r\leq\infty.
\end{equation}

We decompose the nonlinearity as in \eqref{bt}.  We first consider the contribution of $F(u_{\leq M})$ to \eqref{tail1} and \eqref{tail2}, for which it suffices to bound
\begin{equation}\label{12141}
\sum_{M>N}\sum_{k=0}^\infty M^{s_c}\biggl\| \iint \frac{K_M(x,y)}{(M^2|t|)^{100d}}\, \tilde \chi_{k,M} F(u_{\leq M}) \,dy\,dt \biggr\|_{L_x^2}. 
\end{equation}
For this, we use H\"older and Bernstein to estimate
\begin{align*}
\eqref{12141} & \lesssim \sum_{M\geq N}\sum_{k=0}^\infty M^{s_c}(M^2 2^k\delta)^{-100d} \|\tilde \chi_{k,M} F(u_{\leq M})\|_{L_t^1 L_x^2} \\
& \lesssim \sum_{M\geq N}\sum_{k=0}^\infty M^{s_c}(M^2 2^k\delta)^{-100d} 2^k\delta \|u_{\leq M}\|_{L_t^\infty L_x^{2(p+1)}}^{p+1} \\
& \lesssim_u \sum_{M\geq N} \sum_{k=0}^\infty M^{s_c}(M^2 2^k\delta)^{-100d} 2^k\delta M^{\frac{dp}{2}} \lesssim_u (\delta N^2)^{-50d}. 
\end{align*}

We next consider the contribution of $u_L (u_{\leq L})^p$ to \eqref{tail1} and \eqref{tail2}.  We need to estimate
\begin{equation}\label{1219-1}
\begin{aligned}
M^{s_c}&\biggl\| \iint \frac{K_M(x,y)}{(M^2|t|)^{100d}} \tilde\chi_{k,M} u_L (u_{\leq L})^p\,dy\,dt\biggr\|_{L_x^2(|x|\leq N^{-1})} \\
& + M^{s_c}\biggl\| \iint \frac{K_M(x,y)}{(M^2|t|)^{100d}(M|x|)^{\frac{d-1}{2}}} \tilde\chi_{k,M} u_L (u_{\leq L})^p\,dy\,dt\biggr\|_{L_x^2(|x|> N^{-1})}
\end{aligned}
\end{equation}
and then sum over $k\geq 0$, $L\geq 2M$, and $M\geq N$. 

For this we fix an exponent $r$ satisfying
\begin{equation}\label{1221r}
\max\{\tfrac{4-p(d-2)}{p(d+2)},\tfrac{-p^2(d-1)- p(d-3) + 4}{2p}\} < \tfrac 2r< \min\{p,\tfrac{d-1}{d}\}. 
\end{equation}
Such an $r$ exists under the assumption $p_0(d)<p<\frac4d$. 

We further introduce the parameters
\[
r_0 = \tfrac{2r(p+1)}{r+2}\qtq{and} \alpha = \tfrac{d+2}{r_0}-\tfrac{2}{p}.
\]
Note that \eqref{1221r} guarantees that $r_0> 2$, $\alpha>0$, and 
\[
-\tfrac{d}{r_0}<-\alpha<\tfrac{d-1}{2}-\tfrac{d}{r_0},
\]
which is needed when applying the radial Strichartz estimate (cf. Proposition~\ref{P:rad1}). 

Having chosen these parameters, we estimate \eqref{1219-1} using \eqref{schur}, H\"older's inequality, and Lemma~\ref{stb}.  Choosing $\eps$ sufficiently small, we have 
\begin{align*}
\eqref{1219-1}& \lesssim M^{s_c}(M^2 2^k\delta)^{-100d} \bigl\{ \|1\|_{L_x^r(|x|\leq N^{-1})} + \|(M|x|)^{-\frac{d-1}{2}}\|_{L_x^r(|x|>N^{-1})}\bigr\} \\
& \quad \times \biggl\| \int K_M(x,y)\tilde\chi_{k,M} u_L(u_{\leq L})^p \,dy\bigg\|_{L_t^1 L_x^{\frac{2r}{r-2}}} \\
& \lesssim M^{s_c}(M^2 2^k\delta)^{-100d}\bigl\{ N^{-\frac{d}{r}} + (\tfrac{N}{M})^{\frac{d-1}{2}}N^{-\frac{d}{r}}\bigr\} M^{\frac{2d}{r}} \|\tilde \chi_{k,M} u_L (u_{\leq L})^p \|_{L_t^1 L_x^{\frac{2r}{r+2}}} \\
& \lesssim M^{s_c}(M^2 2^k\delta)^{-100d} M^{\frac{2d}{r}}N^{-\frac{d}{r}}\langle 2^k\delta\rangle^{1-\frac{p+1}{r_0}} \|\tilde\chi_{k,M} u_L\|_{L_{t,x}^{r_0}} \|\tilde\chi_{k,M} u\|_{L_{t,x}^{r_0}}^p \\
& \lesssim M^{s_c}(M^2 2^k\delta)^{-100d}M^{\frac{2d}{r}}N^{-\frac{d}{r}}\langle 2^k \delta\rangle^{1-\frac{p+1}{r_0}}(M^2 2^k\delta)^{(p+1)\alpha+\eps} \\ & \quad\times \|\tilde\chi_{k,M} |x|^{-(\alpha-\eps)}u_L\|_{L_{t,x}^{r_0}} \| \tilde\chi_{k,M} |x|^{-\alpha} u\|_{L_{t,x}^{r_0}}^p  \\
& \lesssim_u  M^{s_c}(M^2 2^k\delta)^{-100d}M^{\frac{2d}{r}}N^{-\frac{d}{r}}\langle 2^k \delta\rangle(M^2 2^k\delta)^{(p+1)\alpha+\eps} L^{-\eps}.
\end{align*}
It follows that
\[
\sum_{M\geq N}\sum_{L\geq 2M}\sum_{k=0}^\infty \eqref{1219-1} \lesssim_u (N^2\delta)^{-50d},
\] 
which completes the proof.
\end{proof}

Collecting the results of Lemmas~\ref{L:ST}, \ref{L:MC}, and \ref{L:TE}, we deduce that \eqref{ar-ets} holds.  This completes the proof of Lemma~\ref{lem:addreg} and so proves Theorem~\ref{T:regularity}.
\end{proof}

\section{The cascade and soliton scenarios}\label{S:CS}

In this section, we rule out the cascade scenario of Theorem~\ref{T:reduction}, as well as the soliton scenario in the defocusing case, thereby completing the proofs of Theorems~\ref{T:main} and \ref{T:main3}. The key ingredient in both cases is the additional regularity afforded by Theorem~\ref{T:regularity}.  We will prove that cascade solutions must have zero mass, while soliton solutions are inconsistent with the Lin--Strauss Morawetz inequality in the defocusing case. 

We first rule out the cascade scenario. 
The following theorem completes the proof of Theorem \ref{T:minimizer}.

\begin{theorem}[No cascades]\label{T:no-cascade} 
Let $\mu\in\{\pm 1\}$.  There are no cascade solutions as in Theorem~\ref{T:reduction}.
\end{theorem}

\begin{proof} Suppose toward a contradiction that $u$ is a cascade solution as in Theorem~\ref{T:reduction}.  We will reach a contradiction by proving that $u$ has zero mass.  

Fix $\eta>0$.  By Theorem~\ref{T:regularity}, we know that $u\in L_t^\infty \dot H_x^{|s_c|}$. Thus, using almost periodicity and interpolation and choosing $C=C(\eta)$ sufficiently large, we can first estimate
\[
\|u_{>CN(t)}(t)\|_{L_x^2} \lesssim \| \nsc u_{>CN(t)}\|_{L_t^\infty L_x^2}^{\frac12} \| |\nabla|^{|s_c|}u\|_{L_t^\infty L_x^2}^{\frac12}\lesssim_u \eta^{\frac12},
\]
uniformly for $t\in \R$.  Next, by Bernstein we have
\[
\| u_{\leq CN(t)}(t)\|_{L_x^2} \lesssim [C(\eta)N(t)]^{|s_c|}\| \nsc u\|_{L_t^\infty L_x^2} \lesssim_u [C(\eta)N(t)]^{|s_c|},
\]  
uniformly for $t\in\R$.  Thus, by conservation of mass, we have
\[
\|u(0)\|_{L_x^2} = \|u(t)\|_{L_x^2} \lesssim_u \eta^{\frac12}+[C(\eta)N(t)]^{|s_c|},
\]
uniformly for $t\in\R$.  Applying this to a sequence of times along which $N(t)\to 0$ (recall $u$ is a cascade solution) and noting that $\eta>0$ was arbitrary, we deduce that $u(0)=0$, a contradiction. \end{proof}

We turn to the soliton scenario in the defocusing case.

\begin{theorem}[No solitons]\label{T:no-solitons} Let $\mu=1$. There are no soliton solutions as in Theorem~\ref{T:reduction}.
\end{theorem}

We will need the following standard Morawetz estimate (cf. \cite{LS}).

\begin{lemma}[Lin--Strauss Morawetz estimate]\label{L:morawetz} Let $u:I\times\R^d\to\C$ be a solution to \eqref{nls} with $\mu=1$ in dimension $d\geq 3$.  Then
\begin{equation}\label{morawetz}
\int_I \int_{\R^d} \frac{|u(t,x)|^{p+2}}{|x|}\,dx\,dt \lesssim \| |\nabla|^{\frac12} u\|_{L_t^\infty L_x^2(I\times\R^d)}^2.
\end{equation}
\end{lemma}

Note that by Theorem~\ref{T:regularity}, the right-hand side of \eqref{morawetz} is bounded uniformly in $I\subset\R$ when $u$ is a soliton.  To prove Theorem~\ref{T:no-solitons}, we will prove that the left-hand side of \eqref{morawetz} is bounded below by $|I|$, thus obtaining a contradiction for $I$ long enough. 

\begin{proof}[Proof of Theorem~\ref{T:no-solitons}]
Suppose $u$ is a soliton in the sense of Theorem~\ref{T:reduction}.  As noted above, it suffices to show that the left-hand side of \eqref{morawetz} is bounded below by $|I|$.  To this end, we first observe that since the orbit $\{u(t):t\in\R\}$ is bounded in $H^{1/2}_x$ and precompact in $\dot H^{s_c}$, it is precompact in $L^2$.

On the other hand,
$$
\int_{\R^d}\frac{|u(t,x)|^{p+2}}{|x|} \,dx \geq \int_{\R^d}\frac{\min\{1,|u(t,x)|^{p+2}\}}{\langle x\rangle} \,dx
$$
and the right-hand side is continuous on $L^2$ and vanishes only at $u\equiv 0$.  Thus LHS\eqref{morawetz} is bounded below by a multiple of $|I|$, thereby proving the theorem.
\end{proof}


\end{document}